\documentclass{imanum_X}

\jno{}
\received{\today}
\revised{}

\usepackage{graphicx,graphics,color}

\def\half{\frac{1}{2}}

\newcommand\bfu{{\mathbf u}}

\newcommand\bfw{{\mathbf w}}

\newcommand\bfA{{\mathbf A}}

\newcommand\bfM{{\mathbf M}}

\newcommand\bfW{{\mathbf W}}

\newcommand\calT{{\mathcal T}}
\newcommand\andquad{\quad\hbox{ and }\quad}

\newcommand\quadfor{\quad\hbox{ for }\quad}

\newcommand\quadfore{\quad\hbox{ for every }\quad}
\newcommand{\ccdot}{(\cdot,\cdot)}

\newcommand{\diff}{\frac{\d}{\d t}}
\newcommand{\Half}{{\frac{1}{2}}}

\renewcommand\ss{s}
\renewcommand\d{\hbox{\rm{d}}}

\newcommand{\ga}{\gamma}
\newcommand{\Ga}{\Gamma}

\newcommand{\laplace}{\Delta}

\newcommand{\nbg}{\nabla_{\Gamma}}
\newcommand{\nbgh}{\nabla_{\Gamma_h}}

\newcommand{\nb}{\nabla}
\newcommand{\Om}{\Omega}

\newcommand{\pa}{\partial}
\newcommand{\R}{\mathbb{R}}

\def \t {(t)}

\newcommand{\vphi}{\varphi}

\renewcommand{\nu}{\textnormal{n}}

\newcommand{\bbk}{\color{black}}
\newcommand{\ebk}{\color{black}}

\begin{document}

\title{Error estimates for the Cahn--Hilliard equation \\ with dynamic boundary conditions}
\shorttitle{Error estimates for the Cahn--Hilliard equation with dynamic boundary conditions}

\author{%
	{\sc Paula Harder\thanks{Email: paulaharder@posteo.de}}\\[2pt]
	Fraunhofer Institute for Industrial Mathematics ITWM, \\
	Fraunhofer-Platz 1, 67663 Kaiserslautern, Germany \\[2pt]
	{\sc and} \\[2pt]
	{\sc Bal\'azs Kov\'acs}\thanks{Corresponding author. Email: balazs.kovacs@mathematik.uni-regensburg.de}\\[2pt]
	Faculty of Mathematics, University of Regensburg,\\
	Universit\"atsstra\ss{}e 31, 93049 Regensburg, Germany
	}
\shortauthorlist{P.~Harder and B.~Kov\'acs}

\maketitle

\begin{abstract}
{A proof of convergence is given for a bulk--surface finite element semi-discretisation of the Cahn--Hilliard equation with Cahn--Hilliard-type dynamic boundary conditions in a smooth domain. The semi-discretisation is studied in an abstract weak formulation as a second order system. Optimal-order uniform-in-time error estimates are shown in the $L^2$- and $H^1$-norms. The error estimates are based on a consistency and stability analysis. The proof of stability is performed in an abstract framework, based on energy estimates exploiting the anti-symmetric structure of the second order system. 
Numerical experiments illustrate the theoretical results.}
{Cahn--Hilliard equation; dynamic boundary conditions; bulk--surface finite elements; error estimates; stability; energy estimates.}
\end{abstract}

\section{Introduction}

In the present paper we analyse a bulk--surface finite element discretisation of the Cahn--Hilliard equation with Cahn--Hilliard-type dynamic boundary conditions, introduced in \cite{GoldsteinMiranvilleSchimperna}, in a domain with smooth curved boundary. The studied finite element discretisation is based on the reformulation of the problem into a system of second-order equations. We show optimal-order error estimates that are uniform-in-time for the $L^2$- and $H^1$-norms for both variables, over time intervals where the solutions are sufficiently regular. 
The potentials \bbk and their derivatives \ebk in the equation are only required to satisfy local Lipschitz conditions.

The Cahn--Hilliard equation with Cahn--Hilliard-type dynamic boundary conditions is often used in models which account for interactions with non-permeable walls in a confined system, see, e.g., \cite{Kenzler2001,Racke2003cahn,Goldstein2006derivation,Gal_CHmodel,Gal2008well,Gal2009uniform,GoldsteinMiranvilleSchimperna}, which also contain analytic results for existence, uniqueness and regularity. 
\bbk More general Cahn--Hilliard problems with dynamic boundary condition were studied by  \cite{LiuWu,GarckeKnopf,KnopfLam,KnopfLamLiuMetzger}, and \cite{Miranville_book} (see the references therein as well). These works study a more general dynamic boundary condition with different chemical potentials in the bulk and on the surface, which are coupled in various ways. They prove existence, uniqueness, and well-posedness results, and also address numerical discretisation. For more details we refer to the consciously collected classes in \cite{GarckeKnopf,KnopfLamLiuMetzger}, and also to the references in the above four papers. \ebk

Let us briefly review the numerical analysis literature for problems with dynamic boundary conditions in general.

The paper \cite{ElliottRanner} was the first to analyse isoparametric bulk--surface finite element approximations for an \emph{elliptic} bulk--surface PDE in curved domains, and they have established many fundamental results which are due to the non-conformity of the method. 

A conforming finite element discretisation of the Cahn--Hilliard equation with dynamic boundary conditions in a \emph{rectangle based slab} -- in contrast to our curved domain --  was analysed in \cite{CherfilsPetcuPierre_2010} and \cite{CherfilsPetcu}. In both papers the problem is endowed with dynamic boundary conditions except along the first axis where a \emph{periodic boundary condition} is posed. Both papers use a system formulation of the Cahn--Hilliard equation. The case of Allen--Cahn-type dynamic boundary conditions was analysed in \cite{CherfilsPetcuPierre_2010}, while the case of Cahn--Hilliard-type dynamic boundary conditions was analysed in \cite{CherfilsPetcu}.
In the paper \cite{CherfilsPetcu} optimal-order error bounds for both variables are shown, but -- in contrast to our results -- these estimates are only time-uniform for the original variable and of $L^2$-in-time for the auxiliary variable. The analysis therein assumes global Lipschitz continuity of the derivatives of the non-linearities, and also uses a positivity condition at infinity for their first derivatives. The analysis in the present paper does not use such assumptions.

\bbk 
For the Cahn--Hilliard equation in a \emph{polygonal} domain with the (above mentioned) more general dynamic boundary condition coupling two chemical potentials, finite element semi-discretisation and finite-element/backward Euler full discretisation have been presented and analysed in \cite{KnopfLam,KnopfLamLiuMetzger,Metzger}. In these papers the authors are proving weak convergence of a (sub-)sequence under quite mild regularity assumptions (however, sometimes requiring a CFL-type condition). 
\ebk 

General linear and semi-linear parabolic equations (of second order) with dynamic boundary conditions were analysed in \cite{dynbc}, casting the problems in a general abstract setting and showing optimal-order error bounds for a wide range of problems. The Cahn--Hilliard equation was not analysed therein, although Section~2.3.3 of \cite{dynbc} contains a possible approach. The present paper uses a different one.
Although, \cite{Fairweather} already gave error estimates for a conforming Galerkin method for a class of linear parabolic problems, it went unnoticed in the dynamic boundary conditions community, possibly due to the fact that the term \emph{dynamic} has not appeared at all in his paper.

On the other hand, the numerical analysis of \emph{wave} equations with dynamic boundary conditions on a smooth domain has also developed rapidly over the past few years, see \cite{Hip17,HipHS17,dynbc_waves_L2} for error estimates for linear problems in a bounded domain, \cite{HochbruckLeibold} for semi-linear wave equations (and \cite{acousticGIBC} for wave equations in an unbounded domain).

\bbk To our knowledge, there are no error estimates available for the Cahn--Hilliard equation with Cahn--Hilliard-type dynamic boundary conditions \cite{Gal_CHmodel,GoldsteinMiranvilleSchimperna}. \ebk 

\medskip
In this paper we first rewrite the weak formulation of the second order system corresponding to the Cahn--Hilliard equation with Cahn--Hilliard-type dynamic boundary conditions in a smooth domain in a general abstract setting (very similar to that of \cite{dynbc}). The \emph{anti-symmetric} structure of this abstract weak formulation will play a crucial role later on. The (non-conforming) bulk--surface finite element semi-discretisation is also rewritten in a semi-discrete version of the abstract formulation, which naturally preserves the anti-symmetry.

Our main theorem, stating optimal-order time uniform \bbk $L^2$- and $H^1$-norm \ebk error estimates \bbk of a bulk--surface finite element spatial semi-discretisation for the Cahn--Hilliard equation with Cahn--Hilliard-type dynamic boundary conditions in a \emph{smooth} domain, \ebk will be proved by separately studying the stability and consistency of the method. 

Proving stability, i.e.~a uniform-in-time bound of the errors in terms of the defects and their time derivatives, is the main issue in the paper. The main idea of the stability proof is to exploit the anti-symmetric structure of the semi-discrete error equations and combine it with multiple energy estimates, testing with the errors and also with the time derivative of the errors.  The main idea of the stability proof was originally developed for Willmore flow \cite{Willmore}.
A key issue in the stability proof is to establish uniform-in-time $L^\infty$ norm bounds for the errors, which are then used to estimate the non-linear terms. The $L^\infty$ bounds are obtained from the time-uniform $H^1$ norm error estimates via an inverse estimate. 
The stability proof is completely independent of any geometric approximation errors.

Since the stability proof is entirely performed in an abstract setting, it can be easily generalised to other problems which can be cast in the same setting, for instance the Cahn--Hilliard equation with standard boundary conditions, or more general semi-linear Cahn--Hilliard equations, etc., see Section~\ref{section:generalisations}. The versatility of the stability analysis is clear in view of the related stability proofs in \cite{Willmore} and \cite{CHsurf}.

The consistency analysis, i.e.~proving estimates for the defects (the error obtained upon inserting the Ritz map of the exact solutions into the method) and their time derivatives, uses geometric approximation error estimates and error estimates for the interpolation and the Ritz map for bulk--surface finite elements.

The paper is structured as follows.
In Section~\ref{section:CH} we introduce the basic notation, then formulate the fourth order problem, and rewrite it as a system of second order equations. We formulate the weak formulation corresponding to the system, which is then rewritten in a general abstract setting. In the remainder of the paper we work in this abstract setting.
Section~\ref{section:bulk surface FEM} describes the spatial semi-discretisation using non-conforming bulk--surface finite elements and the abstract semi-discrete setting.
Section~~\ref{section:main results} contains the main result of this paper, Theorem~\ref{theorem:semidiscrete convergence}, which proves optimal-order uniform-in-time error estimates for the semi-discretisation of the Cahn--Hilliard equation with Cahn--Hilliard-type dynamic boundary conditions.
The proof of Theorem~\ref{theorem:semidiscrete convergence} is shown by separating the issues of stability and consistency. Section~\ref{section:stability} contains the stability result based on the novel energy estimates using the anti-symmetric structure of the second order system. Section~\ref{section:generalisations} extends the stability results to various problems. Consistency of the method is studied in Section~\ref{section:consistency} (which also collects various geometric errors bounds and approximation estimates). The combination of the results of these two sections, i.e.~the proof of the main theorem, is performed in Section~\ref{section:proof of main theorem}.
In Section~\ref{section:BDF} we give a brief outlook on linearly implicit backward difference time discretisations, which preserved the mentioned anti-symmetric structure, and are also used in our numerical experiments.
Section~\ref{section:numerics} is devoted to numerical experiments, which illustrate our theoretical results.

\section{Cahn--Hilliard equation with dynamic boundary conditions}
\label{section:CH}

Let us briefly introduce some notations.
Let the bulk $\Om \subset\R^d$ ($d=2$ or $3$) be a bounded domain, with an (at least) $C^2$ boundary $\Ga = \partial\Om$, which is referred to as the surface. Further, let $\nu$ denote the unit outward normal vector to $\Ga$.
Then the surface (or tangential) gradient on $\Ga$, of a function $u : \Ga \to \R$, is denoted by $\nbg u$, and is given by $\nbg u = \nb \bar u -( \nb \bar u \cdot \nu) \nu$, (where $\bar u$ is an arbitrary extension of $u$ in a neighbourhood of $\Ga$), while the Laplace--Beltrami operator on $\Ga$ is given by $\laplace_\Ga u = \nbg \cdot \nbg u$. For more details see, e.g., \cite{DziukElliott_acta}.
Moreover, $\ga u$ denotes the trace of $u$ on $\Ga$ 
and $\pa_{\nu} u$ denotes the normal derivative of $u$ on $\Ga$.
Finally, temporal derivatives are denoted by $\dot{\phantom{u}} = \d / \d t$.

In this paper we consider the Cahn--Hilliard equation with dynamic boundary conditions of Cahn--Hilliard-type first derived in \cite{GoldsteinMiranvilleSchimperna}, or -- as we will also refer to it -- the Cahn--Hilliard/Cahn--Hilliard coupling, that is the fourth-order equation, for a function $u: \overline{\Om} \times [0,T] \to \R$,
\begin{subequations}
	\label{eq:CH - 4th order form}
	\begin{alignat}{3}
	\label{eq:CH - 4th order form - eq}
	\dot u &= \laplace \big( -\laplace u + W_\Om'(u) \big) & \qquad & \text{in } \Om, \\
	\label{eq:CH - 4th order form - bc}
	\dot u &= \laplace_\Ga \big( -\laplace_\Ga u + W_\Ga'(u)  + \pa_\nu u \big) - \pa_\nu \big( -\Delta_\Ga u + W'_\Ga(u)  + \pa_\nu u \big) & \qquad & \text{on }  \Ga ,
	\end{alignat}
\end{subequations}
with continuous (and sufficiently regular) initial condition $u(0) = u^0$. The scalar functions $W_\Om$ and $W_\Ga$ are \bbk free energy \ebk potentials and are only assumed to have locally Lipschitz \bbk second and third \ebk derivatives. By $W'$ we simply denote the derivative of $W$.

In typical examples double well potentials are often used, i.e.~$W(u) = (u^2 - 1)^2$, \bbk for which the above assumptions are satisfied. \ebk The solution $u \in [-1,1]$ models the concentration of two fluids, with $u=\pm 1$ indicating the pure occurrences of each. 
To model the interactions within systems confined by a non-permeable wall, dynamic boundary conditions in this context were introduced in \cite{GoldsteinMiranvilleSchimperna}.

\subsection{Weak formulation as a second order system}
\label{section:weak form}
By introducing an auxiliary function $w:\overline{\Om} \times [0,T] \to \R$ we rewrite the Cahn--Hilliard equation \eqref{eq:CH - 4th order form} into a system of second order partial differential equations: For $u,w:\overline{\Om} \times [0,T] \to \R$
\begin{subequations}
	\label{eq:CH system}
	\begin{alignat}{3}
	\label{eq:CH system - u Om}
	\dot u = &\ \Delta w & \quad & \text{in } \Om \\
	\label{eq:CH system - w Om}
	w = &\  -\Delta u + W_\Om'(u) & \quad & \text{in }  \Om , \\
	\intertext{with dynamic Cahn--Hilliard boundary conditions}
	\label{eq:CH system - u Ga}
	\dot u = &\ \Delta_\Ga w - \partial_\nu w & \quad & \text{on } \Ga \\
	\label{eq:CH system - w Ga}
	w = &\  -\Delta_\Ga u + W'_\Ga(u) + \partial_\nu u & \qquad & \text{on } \Ga .
	\end{alignat}
\end{subequations}

\bbk Let us note here that for the more general Chan--Hilliard/Cahn--Hilliard coupling of \cite{GarckeKnopf,KnopfLamLiuMetzger,KnopfLam} is the Cahn--Hilliard equation but on the boundary \eqref{eq:CH system - u Ga}--\eqref{eq:CH system - w Ga} holds for the chemical potential $w_\Ga$ instead of $w$, which are then coupled by $L \pa_{\nu} w = \beta w_\Ga - w$ (with parameters $L>0$, $\beta \neq 0$). The coupled problem \eqref{eq:CH system} is the limit $L \searrow 0$ and $\beta = 1$. \ebk 

Before we turn to the weak formulation of the problem \eqref{eq:CH system}, let us introduce some notations for function spaces. We will use standard Sobolev spaces $H^k(\Om)$ and $H^k(\Ga)$, for $k \geq 0$, in the bulk and on the surface, respectively, for more details see \cite{DziukElliott_acta} and \cite{ElliottRanner}. The variational formulation will also use the Hilbert spaces 
\begin{equation}
\label{eq:Hilbert spaces}
	\begin{aligned}
		V = &\ \{ v \in H^1(\Om) \mid \ga u \in H^1(\Ga) \} 
		\andquad 
		H = L^2(\Om) \times L^2(\Ga) , \\
		&\ \textnormal{with a dense embedding from $V$ into $H$:} \quad v \mapsto (v, \ga v) ,
	\end{aligned} 
\end{equation}
with norms
\begin{equation}
\label{eq:Hilbert space norms}
	\begin{aligned}
		\|u\|^2 := &\ \|u\|_V^2 = \|u\|_{H^1(\Om)}^2 + \|\ga u\|_{H^1(\Ga)}^2 
		\andquad \\
		\bbk |(u_\Om,u_\Ga)|^2 := \ebk &\ \bbk \|(u_\Om,u_\Ga)\|_H^2 = \|u_\Om\|_{L^2(\Om)}^2 + \|u_\Ga\|_{L^2(\Ga)}^2 \ebk  .
	\end{aligned} 
\end{equation}
A similar setting was used in \cite{dynbc} and \cite{Hip17,dynbc_waves_L2} for parabolic and wave-type problems with dynamic boundary conditions. We will abbreviate pairs $(v,\ga v)$ in $H$ by their first argument $v$.

The weak formulation of the Cahn--Hilliard equation with dynamic boundary conditions of Cahn--Hilliard-type is derived by multiplying \eqref{eq:CH system - u Om} with the test function $\vphi^u \in V$, and \eqref{eq:CH system - w Om} with $\vphi^w \in V$, cf.~\cite[equation~(3.9)--(3.10)]{GoldsteinMiranvilleSchimperna}. Then both equations are integrated over the domain $\Om$, and by applying Green's formula for both, we obtain
\begin{align*}
	\int_\Om \dot u \vphi^u = &\ - \int_\Om  \nb w \cdot \nb \vphi^u + \int_\Ga \partial_\nu w \gamma \vphi^u , \\
	\int_\Om w \vphi^w = &\ \int_\Om \nb u\cdot \nb\vphi^w+ \int_\Om W'_\Om(u) \vphi^w - \int_\Ga \partial_\nu u \gamma \vphi^w.
\end{align*}
Plugging in the dynamic boundary conditions \eqref{eq:CH system - u Ga} and \eqref{eq:CH system - w Ga} into the boundary terms above, applying Green's formula now on the boundary, and collecting the terms, yield
\begin{subequations}
	\label{eq:weak form}
	\begin{align}
	\label{eq:weak form - u}
	&\ \Big(\int_\Om \dot u \vphi^u + \int_\Ga\gamma \dot u \gamma \vphi^u\Big) +\Big( \int_\Om  \nb w\cdot \nb \vphi^u + \int_\Ga \nb_\Ga  w\cdot \nb_\Ga  \vphi^u \Big) = 0 , \\
	&\ \Big(\int_\Om   w \vphi^w+\int_\Ga \gamma w \gamma \vphi^w \Big) -\Big(\int_\Om \nb u\cdot \nb\vphi^w + \int_\Ga \nb_\Ga u \cdot \nb_\Ga \vphi^w \Big) 
	\label{eq:weak form - w}
	= \int_\Om W'_\Om(u) \vphi^w + \int_\Ga W'_\Ga(u) \gamma \vphi^w ,
	\end{align}
\end{subequations}
where for brevity we write $\nbg v$ instead of $\nbg (\ga v)$, and have also suppressed the trace operator in $\bbk W'_\Ga(u) := \ebk W'_\Ga(\bbk \ga \ebk u)$. We will employ these conventions throughout the paper.

\subsection{Abstract formulation}
\label{section:abstract formulation}

It is insightful to formulate the variational formulation in an abstract setting.
To this end we introduce the bilinear forms, on $V$ and $H$, respectively, 
\begin{equation}
\label{eq:bilinear forms}
\begin{aligned}
a(u,v) = &\ \int_\Om \nb u \cdot \nb v + \int_\Ga \nbg u \cdot \nbg v 
\andquad
\bbk m\big((u_\Om,u_\Ga) , (v_\Om,v_\Ga) \big) = \int_\Om u_\Om v_\Om + \int_\Ga u_\Ga v_\Ga \ebk , 
\end{aligned}
\end{equation}
and let
\begin{equation*}
a^*\ccdot = a\ccdot + m\ccdot .
\end{equation*}
We note here that the norms \eqref{eq:Hilbert space norms} on $V$ and $H$ are directly given by 
\begin{align*}
	\|u\|^2 = &\ a^*(u,u) = a(u,u) + m((u,\ga u) , (v,\ga v)) \andquad
	\bbk |(u_\Om,u_\Ga)|^2 \bbk =  m((u_\Om,u_\Ga) , (v_\Om,v_\Ga))  \ebk .
\end{align*}
Furthermore, on $V$ the bilinearform $a\ccdot$ generates the semi-norm
\begin{equation*}
	\|u\|_a^2 := a(u,u) .
\end{equation*}
\bbk 
We will mostly work with elements of $V$ embedded into $H$, and according to the above notational convention $u = (u,\ga u) \in H$ for embedded pairs, for brevity we will write 
\begin{equation*}
	m(u,v) = m\big( (u,\ga u) , (v,\ga v) \big), \andquad |u|^2 = m(u,u) .
\end{equation*}
These notations are also employed throughout the paper. 
\ebk

For the non-linear term, under the expression $m(W'(u),\vphi^w)$ we mean 
\begin{equation}
\label{eq:nonlinear term def}
	m(W'(u),\vphi^w) = \int_\Om W'_\Om(u) \vphi^w + \int_\Ga W'_\Ga(u) \gamma \vphi^w.
\end{equation}
Using these bilinear forms, we rewrite the weak formulation \eqref{eq:weak form}. The Cahn--Hilliard equation with Cahn--Hilliard-type dynamic boundary conditions in the above abstract setting reads: Find a function $u\in C^1([0,T],H) \cap L^2([0,T],V)$ and $w \in L^2([0,T],V)$ such that, for time $0 < t\leq T$ and all $\vphi^u, \vphi^w \in V$,
\begin{subequations}
	\label{eq:abstract weak form}
	\begin{alignat}{2}
	\label{eq:abstract weak form - u}
	m(\dot u\t,\vphi^u) + a(w\t,\vphi^u) = &\ 0 , \\
	\label{eq:abstract weak form - w}
	m(w\t,\vphi^w) - a(u\t,\vphi^w) = &\ m(W'(u\t),\vphi^w) ,
	\end{alignat}
\end{subequations}
for given initial data \bbk $u(0) = u^0 \in V$, with finite Ginzburg--Landau energy. \ebk 

From now on we will use this (rather general) abstract formulation for the Cahn--Hilliard equation with Cahn--Hilliard-type dynamic boundary condition. 

\begin{remark}
	\label{remark:other dynbc}
	Other types of dynamic boundary conditions also fit into this framework by using different Hilbert spaces, and changing the boundary integrals in the bilinear forms $m$ and $a$. For more details see Section~\ref{section:generalisations}.
\end{remark}

\bbk 
It is important to note here that the weak form of the standard Cahn--Hilliard equation fits into the above abstract framework, and its weak formulation can then be written exactly as \eqref{eq:abstract weak form}. Therefore, apart from the well-posedness and regularity results of \cite{GoldsteinMiranvilleSchimperna}, \cite{Miranville,Miranville_book}, many results for Cahn--Hilliard equations apply to the present case as well \cite[Chapter~3]{Miranville_book}, or \cite{ElliottRanner_CH}, etc.

We directly obtain, by testing \eqref{eq:abstract weak form - u} with $w$ and \eqref{eq:abstract weak form - w} with $\dot u$, and combining the two equalities to cancel the mixed terms $m(\dot u,w)$, the energy decay
\begin{equation}
\label{eq:the first energy estimate}
	0 = a(u,\dot u) + m(W'(u),\dot u) + a(w,w) 
	= \half \diff  \Big( \|u\|_a^2 + m(W(u),1) \Big) + \|w\|_a^2 ,
\end{equation}
and therefore the Ginzburg--Landau free energy
\begin{equation}
\label{eq:energy definition}
	\mathcal{E}(t) := \int_\Om |\nb u(\cdot,t)|^2 + \int_\Ga |\nbg (\ga u(\cdot,t))|^2 + \int_\Om W_\Om(u(\cdot,t))  + \int_\Ga W_\Ga(\ga u(\cdot,t)) 
\end{equation}
is monotone decreasing in $[0,T]$.

The following well-posedness and regularity result was proved in \cite{GoldsteinMiranvilleSchimperna}, in particular see Theorem~3.2 and 3.6, with Remark~3.13: For an initial value $u^0$ of finite energy $\mathcal{E}(0)$ which is smooth enough that $w(\cdot,t) \in V$ (via \eqref{eq:CH system - w Om} and \eqref{eq:CH system - w Ga}), then the solution satisfies
\begin{equation}
\label{eq:reg - GMS}
	\begin{aligned}
		&\ u \in L^\infty(0,T;V) \qquad \text{with} \qquad \dot u \in L^2(0,T;V) , \\
		&\ w \in L^2(0,T;V) \cap L^\infty(0,T;V) .
	\end{aligned}
\end{equation}
The straightforward adaptation of the arguments of \cite[Section~4.2--4.4]{ElliottRanner_CH} yields the $H^2$ regularity result, for $u^0 \in H^2(\Om)$ with $\ga u^0 \in H^2(\Ga)$:
\begin{equation}
\label{eq:reg - ER adapt}
\begin{aligned}
	&\ u \in L^\infty(0,T;H^2(\Om)) \quad \text{with} \quad \ga u \in L^\infty(0,T;H^2(\Ga)) , \\
	&\ w \in L^2(0,T;H^2(\Om)) \quad \text{with} \quad \ga w \in L^2(0,T;H^2(\Ga)) .
\end{aligned}
\end{equation}

\ebk 

\section{Semi-discretisation of Cahn--Hilliard equations with dynamic boundary conditions}
\label{subsection:semi discrete abstract CH equations}

For the numerical solution of the above examples we consider a linear finite element method both in the bulk and on the surface. 
In the following, from \cite{ElliottRanner}, \cite[Section~3.2.1]{dynbc}, and \cite{dynbc_waves_L2}, we will briefly recall the construction of the discrete domain, the finite element space and the lift operation, the discrete bilinear forms, which will be used to discretize the Cahn--Hilliard problem of Section~\ref{section:CH}.

\subsection{The bulk--surface finite elements}
\label{section:bulk surface FEM}

The domain $\Om$ is approximated by a triangulation $\calT_h$ with \mbox{maximal mesh width $h$.}
The union of all elements  of $\calT_h$ defines the polyhedral domain $\Om_h$ whose boundary $\Ga_h := \pa \Om_h$ is an interpolation of $\Ga$, i.e.~the vertices of $\Ga_h$ are on $\Ga$.
Analogously, we denote the outer unit normal vector of $\Ga_h$ by $\nu_h$.
We assume that $h$ is sufficiently small to ensure that for every point $x\in\Ga_h$ there is a unique point $p\in\Ga$ such that $x-p$ is orthogonal to the tangent space $T_p\Ga$ of $\Ga$ at $p$.
For convergence results, we consider a quasi-uniform family of such triangulations $\calT_h$ of $\Om_h$.

The finite element space $V_h \nsubseteq H^1(\Om)$ corresponding to $\calT_h$ is spanned by continuous, piecewise linear nodal basis functions on $\Omega_h$, satisfying for each node $(x_k)_{k=1}^N$
$$
\phi_j(x_k) = \delta_{jk}, \quadfor j,k = 1, \dotsc, N .
$$
Then the finite element space is given as
$$
V_h = \textnormal{span}\{\phi_1, \dotsc, \phi_N \} .
$$
We note here that the restrictions of the basis functions to the boundary $\Ga_h$ again form a surface finite element basis over the approximate boundary elements.

\bbk The discrete tangential gradient $\nbgh$ is piece-wisely defined analogously to $\nbg$. The discrete trace operator $\ga_h v_h$ is defined by the restriction of the continuous function $v_h \in V_h$ onto $\Ga_h$. We apply the same notational conventions as for the continuous case, e.g.~$\nbgh v_h := \nbgh (\ga_h v_h)$, etc. \ebk 

Following \cite{Dziuk88}, we define the \emph{lift} operator $\cdot^\ell \colon V_h \to V$ to compare functions in $V_h$ with functions in $V$. For functions $v_h:\Ga_h\to \R$, we define the lift as 
\begin{equation}
\label{eq:lift definition}
v_h^\ell \colon \Ga\to\R \quad \text{with} \quad v_h^\ell(p)=v_h(x), \quad \forall p\in\Gamma,
\end{equation} 
where $x\in\Ga_h$ is the \emph{unique} point on $\Ga_h$ with $x-p$  orthogonal to the tangent space $T_p\Ga$.
We further consider the \emph{lift} of functions $v_h:\Om_h\to \R$ to $v_h^\ell:\Om\to\R$ by setting $v_h^\ell(p)=v_h(x)$  if $x\in\Omega_h$ and $p\in \Omega$, where the two points are related as described in detail in \cite[Section~4]{ElliottRanner}. 
The mapping $G_h : \Om_h \to \Om$ is defined piecewise, for an element $E \in \calT_h$, by
\begin{equation}
\label{eq:bulk mapping}
	G_h|_E (x) = F_e\big((F_E)^{-1}(x)\big), \qquad \text{for } x \in E,
\end{equation}
where $F_e$ is a $C^1$ map (see \cite[equation~(4.2) \& (4.4)]{ElliottRanner}) from the reference element onto the smooth element $e \subset \Om$, and $F_E$ is the standard affine linear map between the reference element and $E$, see, e.g.~\cite[equation~(4.1)]{ElliottRanner}. 
Finally, the lifted finite element space is denoted by $V_h^\ell$, and is given as $V_h^\ell = \{  v_h^\ell \mid v_h \in V_h \}$.

Note that both definitions of the lift coincide on $\Gamma$.

\subsection{Discrete spaces and discrete bilinear forms}

The discrete bilinear forms on $V_h$, i.e.~the discrete counterparts of $a$ and $m$, are given, for $u_h,v_h \in V_h$, by
\begin{equation}
\label{eq:discrete bilinear forms}
\begin{aligned}
a_h(u_h,v_h) = &\ \int_{\Om_h} \nb u_h \cdot \nb v_h + \int_{\Ga_h} \nbgh u_h \cdot \nbgh v_h , \\
m_h(u_h,v_h) = &\ \int_{\Om_h} u_h v_h + \int_{\Ga_h} (\ga_h u_h) (\ga_h v_h) ,
\end{aligned}
\end{equation}
and let $a_h^*\ccdot = a_h\ccdot + m_h\ccdot$.

The discrete norms on $V_h$, corresponding to $\|\cdot\|$ and $|\cdot|$, are given by 
\begin{align*}
\|u_h\|_h^2 := &\ \|u_h\|_{H^1(\Om_h)}^2 + \|\ga_h u_h\|_{H^1(\Ga_h)}^2 = a_h(u_h,u_h) + m_h(u_h,u_h) , \\
|u_h|_h^2 := &\ \|u_h\|_{L^2(\Om_h)}^2 + \|\ga_h u_h\|_{L^2(\Ga_h)}^2 = m_h(u_h,u_h) ,
\intertext{and the discrete semi-norm induced by $a_h\ccdot$ (analogously to the continuous case)}
\|u_h\|_{a_h}^2 := &\ \|\nb u_h\|_{L^2(\Om_h)}^2 + \|\nbgh u_h\|_{L^2(\Ga_h)}^2 = a_h(u_h,u_h) .
\end{align*}

Later on in Section~\ref{section:geometric approx errors}, it will be shown that these discrete norms and their continuous counterparts are $h$-uniformly equivalent.

\subsection{A Ritz map}

We will now define a Ritz map, which will be used in the error analysis, and also for prescribing the initial data for the semi-discrete problem. 

From \cite[Section~3.4]{dynbc} we recall the Ritz map $\widetilde R_h : V \to V_h$ which is defined, for $u\in V$, by
\begin{equation}
\label{eq:Ritz map definition}
a_h^*(\widetilde R_h u,\vphi_h) = a^*(u,\vphi_h^\ell), \quadfore \vphi_h \in V_h .
\end{equation}
The above Ritz map is well-defined for all $u \in V$ due to the ellipticity of the bilinear form $a_h^*$, see \cite[Section~3.4]{dynbc}. Note that the Ritz map $\widetilde R_h u$ is the Riesz representation of $u$.  Further, note that, the bilinear forms $a^*$ and $a_h^*$ contain boundary integrals which influence $\widetilde R_h$.
The lifted Ritz map will be denoted by $R_h u := (\widetilde R_h u)^\ell \in V_h^\ell$.

\subsection{Semi-discrete problem}
\label{section:semidiscrete problem}

The semi-discrete problem then reads: Find $u_h \in C^1([0,T],V_h)$ and $w_h \in L^2([0,T],V_h)$ such that, for time $0 < t\leq T$ and all $\vphi_h^u, \vphi_h^w \in V$,
\begin{subequations}
	\label{eq:semi-discrete problem - pre}
	\begin{alignat}{2}
	m_h(\dot u_h\t, \vphi_h^u) + a_h(w_h\t,\vphi_h^u) = &\ 0 , \\
	\label{eq:semi-discrete problem - b - pre}
	m_h(w_h\t,\vphi_h^w) - a_h(u_h\t,\vphi_h^w) = &\ m_h(W'(u_h\t),\vphi_h^w) ,
	\end{alignat}
\end{subequations}
where the initial data $u_h(0) = u_h^0 \in V_h$ is chosen to be the Ritz map of $u^0$ \bbk (required to be in $V$), \ebk i.e.~$u_h(0) = u_h^0 = \widetilde R_h u^0 \in V_h$, and the initial value for $w_h$ is obtained by solving the elliptic equation \eqref{eq:semi-discrete problem - b - pre}.

\bbk 
According to the following result the semi-discrete problem is well-posed, and both solution components have continuous time-derivatives. Similar results have been obtained for Cahn--Hilliard equations, e.g., with dynamic boundary conditions in a rectangular based time-slab \cite[Theorem~2.1]{CherfilsPetcu}, or on an evolving surface \cite[Theorem~3.1]{ElliottRanner_CH}. 
\begin{proposition}
\label{proposition:semi-discrete well-posedness}
	In the above semi-discrete setting, including the assumption on the mesh, initial data, and nonlinearities, the semi-discrete problem \eqref{eq:semi-discrete problem - pre} has a unique solution $u_h(\cdot,t), w_h(\cdot,t) \in C^1([0,T],V_h)$
	and which has monotone decreasing Ginzburg--Landau energy $\mathcal{E}_h(t)$, satisfying the semi-discrete energy estimate, for $0 \leq t \leq T$,
	\begin{equation*}
		\mathcal{E}_h(t) := \int_{\Om_h} |\nb u_h(\cdot,t)|^2 + \int_{\Ga_h} |\nbgh (\ga_h u_h(\cdot,t))|^2 + \int_{\Om_h} W_\Om(u_h(\cdot,t))  + \int_{\Ga_h} W_\Ga(\ga_h u_h(\cdot,t)) \leq  \mathcal{E}_h(0) .
	\end{equation*}
	In addition, if $u^0 \in H^2(\Om)$ with $\ga u^0 \in H^2(\Ga)$, then the energy satisfies $\mathcal{E}_h(t) \leq c \mathcal{E}(0)$ for any $0 \leq t \leq T$.
\end{proposition}

\begin{proof}
The proof uses standard theory of ordinary differential equations to show short-time existence, and extends this to $[0,T]$ by energy estimates. The proof is based on the proof (for a more complicated problem) of Theorem~3.1 in \cite{ElliottRanner_CH}.

We first rewrite the semi-discrete problem \eqref{eq:semi-discrete problem - pre} in a matrix--vector formulation. Collecting the nodal values of $u_h(\cdot,t)$ and $w_h(\cdot,t)$ into $\bfu\t \in \R^N$ and $\bfw\t \in \R^N$, respectively, and using the mass matrix $\bfM|_{ij} = m(\phi_j,\phi_i)$, stiffness matrix $\bfA|_{ij} = a(\phi_j,\phi_i)$, and non-linear term $\bfW'(\bfu\t)|_j = m_h(W'(u_h(\cdot,t)),\phi_j)$ the problem \eqref{eq:semi-discrete problem - pre} is equivalently written as:
\begin{align*}
	\bfM \dot\bfu\t + \bfA \bfw\t = &\ 0 , \\
	\bfM \bfw\t - \bfA \bfu\t = &\ \bfW'(\bfu\t) .
\end{align*}
Therefore, $\bfu\t$ is the solution of the ODE system $\bfM \dot\bfu\t + \bfA \bfM^{-1} (\bfA \bfu\t + \bfW'(\bfu\t)) = 0$. Since $\bfW'$ is locally Lipschitz continuous by assumption, by standard ODE theory there exists a unique short-time solution $\bfu\t \in C^1([0,T_0],\R^N)$. Therefore, using that 
$$
	\bfw\t = \bfM^{-1} \big(\bfA \bfu\t + \bfW'(\bfu\t) \big) ,
$$
we obtain that $\bfw\t$ is also $C^1$ in time.
	
To extend the solutions until the final time $T$ we use an energy bound. By the analogous argument which lead to \eqref{eq:the first energy estimate}, we obtain the analogous semi-discrete energy decay, and the energy estimate $\mathcal{E}_h(t) \leq c \mathcal{E}_h(0)$, which is then used to extend the existence of solutions onto $[0,T]$.
Translating these results back to the functional analytic setting, yields the stated results.

The final estimate between the initial energies hold by error estimates for the Ritz map \cite[Lemma~3.11 and 3.15]{dynbc}.
\end{proof}

\ebk

\subsection{A modified semi-discrete problem}
\label{section:modified semidiscrete problem}

Since the initial value for $w_h$ is not freely chosen, but is determined by the equations, it is to be expected that they will be involved in the error analysis. Our error analysis will provide an optimal order error bound in both norms $\|\cdot\|$ and $|\cdot|$, provided that this initial value is $O(h^2)$ close to the Ritz map of the exact initial value in the stronger norm. 

Since such an $O(h^2)$ estimate can only be verified in a weaker norm, we modify the second equation \eqref{eq:semi-discrete problem - b - pre} with a time-independent term to obtain optimal-order error estimates. The additional term corrects the initial value $\bar{w}_h(0) \in V_h$ (obtained from \eqref{eq:semi-discrete problem - b - pre} at $t=0$) to coincide with the Ritz map of the exact initial value $w_h^*(0) := \widetilde R_h w(0) \in V_h$.

Let $\vartheta_h \in V_h$ be given by
\begin{equation}
\label{eq:def theta}
	\vartheta_h = w_h^*(0) - \bar{w}_h(0),
\end{equation}
then the modified semi-discrete problem then reads: Find $u_h \in C^1([0,T],V_h)$ and $w_h \in L^2([0,T],V_h)$ such that, for time $0 < t\leq T$ and all $\vphi_h^u, \vphi_h^w \in V$,
\begin{subequations}
	\label{eq:semi-discrete problem}
	\begin{alignat}{2}
		\label{eq:semi-discrete problem - a}
		m_h(\dot u_h\t, \vphi_h^u) + a_h(w_h\t,\vphi_h^u) = &\ 0 , \\
		\label{eq:semi-discrete problem - b}
		m_h(w_h\t,\vphi_h^w) - a_h(u_h\t,\vphi_h^w) = &\ m_h(W'(u_h\t),\vphi_h^w) + m_h(\vartheta_h,\vphi_h^w) ,
	\end{alignat}
\end{subequations}

The initial data for \eqref{eq:semi-discrete problem - a} is still $u_h(0) = u_h^0 = \widetilde R_h u^0 \in V_h$.

\bbk Well-posedness results analogous to Proposition~\ref{proposition:semi-discrete well-posedness} hold for the modified semi-discrete problem \eqref{eq:semi-discrete problem} as well. \ebk 

The initial data for $w_h$, obtained from \eqref{eq:semi-discrete problem - b}, is 
\begin{align*}
	m_h(w_h(0),\vphi_h^w) = &\ a_h(u_h(0),\vphi_h^w) + m_h(W'(u_h(0)),\vphi_h^w) + m_h(\vartheta_h,\vphi_h^w) \\
	= &\ m_h( w_h^*(0) ,\vphi_h^w) ,
\end{align*}
via \eqref{eq:semi-discrete problem - b - pre}.

The advantage of the modified system, is that the errors in the initial data for $w_h$ are included into the problem similarly to a defect, which allows for feasible weaker norm estimate of this error. Note that for the linear case, this is nothing else but shifting the solutions to a particular initial value using a constant inhomogeneity.

\section{Main results: optimal-order semi-discrete error estimates}
\label{section:main results}

We are now able to state the main theorem of this paper.

\begin{theorem}
	\label{theorem:semidiscrete convergence}
	Let $u$ and $w$ be sufficiently smooth solutions to the Cahn--Hilliard equation with Cahn--Hilliard-type dynamic boundary conditions \eqref{eq:CH system}, sufficient regularity assumptions are \eqref{eq:sufficient regularity}.
	
	Then there exists an $h_0 > 0$ such that for all $h \leq h_0$ the error between the solutions $u$ and $w$ and the linear finite element semi-discretisations $u_h$ and $w_h$ of \eqref{eq:semi-discrete problem} satisfy the optimal-order uniform-in-time error estimates in both variables, for $0 \leq t \leq T$,
	\begin{alignat*}{3}
		&\ \|u_h^\ell(\cdot,s) - u(\cdot,s)\|_{L^2(\Om)}  + \|\ga (u_h^\ell(\cdot,s) - u(\cdot,s))\|_{L^2(\Ga)} \\
		&\ \quad + h  \Big( \|u_h^\ell(\cdot,s) - u(\cdot,s)\|_{H^1(\Om)}  + \|\ga (u_h^\ell(\cdot,s) - u(\cdot,s))\|_{H^1(\Ga)} \Big) \leq C h^2, \\
		&\ \|w_h^\ell(\cdot,s) - w(\cdot,s)\|_{L^2(\Om)}  + \|\ga (w_h^\ell(\cdot,s) - w(\cdot,s))\|_{L^2(\Ga)} \\
		&\ \quad + h  \Big( \|w_h^\ell(\cdot,s) - w(\cdot,s)\|_{H^1(\Om)}  + \|\ga (w_h^\ell(\cdot,s) - w(\cdot,s))\|_{H^1(\Ga)} \Big) \leq C h^2 , 
	\end{alignat*}
	whereas for the time derivatives of the errors in $u$ satisfy, for $0 \leq t \leq T$,
	\begin{equation*}
		\begin{aligned}
			&\ \bigg( \int_0^t \big\| \partial_t (u_h^\ell(\cdot,s) - u(\cdot,s)) \big\|_{L^2(\Om)}^2  
			+ \big\| \partial_t (\ga (u_h^\ell(\cdot,s) - u(\cdot,s))) \big\|_{L^2(\Ga)}^2 \\
			&\ \quad + h  \Big( \big\| \partial_t (u_h^\ell(\cdot,s) - u(\cdot,s)) \big\|_{H^1(\Om)}^2  
			+ \big\| \partial_t (\ga (u_h^\ell(\cdot,s) - u(\cdot,s))) \big\|_{H^1(\Ga)}^2 \Big) \d s \bigg)^{1/2} \leq C h^2 .
		\end{aligned}
	\end{equation*}
	The constant $C > 0$ depends on Sobolev norms of the solutions, and \bbk exponentially \ebk on the final time $T$, but it is independent of $h$ and $t$.
\end{theorem}

\bbk 
For Theorem~\ref{theorem:semidiscrete convergence} sufficient regularity conditions are
\begin{subequations}
\label{eq:sufficient regularity}
\begin{equation}
\label{eq:sufficient regularity - a}
	\begin{aligned}
		u \in &\ C^1( [0,T],H^1(\Om) ) \cap H^2( [0,T],H^2(\Om) ) 
		\quad \text{with} \\
		\ga u \in &\ C^1( [0,T],H^1(\Ga)) \cap H^2( [0,T],H^2(\Ga) ) 
		, \\ \text{and} \\
		w \in &\ C( [0,T],H^1(\Om) ) \cap H^1( [0,T],H^2(\Om) ) \quad \text{with} \\
		\ga w \in &\ C( [0,T],H^1(\Ga) ) \cap H^1( [0,T],H^2(\Ga) ) .
	\end{aligned}
\end{equation}
On one hand, by standard theory, for $u \in H^2([0,T];X)$, the estimate $\max_{0\leq t \leq T} (\|u(t)\|_{X} + \|\dot u(t)\|_{X} ) \leq c \|u\|_{H^2(0,T;X)}$ holds, see, e.g.~\cite[Section~5.9.2]{Evans_PDE}; on the other hand the embedding $H^2 \subset L^\infty$ is continuous as well.
Therefore, \eqref{eq:sufficient regularity - a} implies that
\begin{equation}
\label{eq:sufficient regularity - b}
	\begin{alignedat}{3}
		%
		u \in &\ W^{1,\infty}( [0,T],L^\infty(\Om) ) 
		\quad \text{with} & \quad
		\ga u \in &\ W^{1,\infty}( [0,T],L^\infty(\Ga) ) .
	\end{alignedat}
\end{equation}
\end{subequations}
\ebk 

Our main theorem, Theorem~\ref{theorem:semidiscrete convergence}, will be proved by separately studying the questions of stability and consistency in Section~\ref{section:stability} and \ref{section:consistency}, respectively, and combining their results in Section~\ref{section:proof of main theorem}.

\section{Stability}
\label{section:stability}

%

\subsection{Error equations}
\label{section:error equations}

Let us consider the Ritz map of the exact solutions $u$ and $w$ of \eqref{eq:CH system}, which are denoted by
\begin{equation*}
u_h^*\t = \widetilde R_h u\t \in V_h, \andquad w_h^*\t = \widetilde R_h w\t \in V_h .
\end{equation*}
The Ritz maps of the exact solutions satisfy the system \eqref{eq:semi-discrete problem - pre} only up to some defects, \bbk i.e.~the bulk--surface finite element residuals, \ebk $d_h^u$ and $d_h^w$ in $V_h$:
\begin{subequations}
	\label{eq:defect definition}
	\begin{align}
	\label{eq:defect definition - u}
	m_h(\dot u^*_h\t, \vphi_h^u) + a_h(w^*_h\t,\vphi_h^u) = &\ m_h(d_h^u\t,\vphi_h^u) , \\
	\label{eq:defect definition - w}
	m_h(w_h^*\t,\vphi_h^w) - a_h(u_h^*\t,\vphi_h^w) = &\ m_h(W'(u_h^*\t),\vphi_h^w) + m_h(d_h^w\t,\vphi_h^w),
	\end{align}
\end{subequations}
with initial data as the Ritz map of the exact initial data, i.e.~$u_h^*(0) = \widetilde R_h u^0$.

The errors between the semi-discrete solutions and the Ritz maps of the exact solutions are denoted by $e_h^u = u_h - u_h^*$ and $e_h^w = w_h - w_h^*$ in $V_h$. 
By subtracting \eqref{eq:defect definition} from \eqref{eq:semi-discrete problem} we obtain that the errors $e_h^u$ and $e_h^w$ satisfy the following error equations:
\begin{subequations}
	\label{eq:error equations}
	\begin{align}
	\label{eq:error equations - u}
	m_h(\dot e_h^u\t, \vphi_h^u) + a_h(e_h^w\t,\vphi_h^u) = &\ - m_h(d_h^u\t,\vphi_h^u) , \\
	m_h(e_h^w\t,\vphi_h^w) - a_h(e_h^u\t,\vphi_h^w) = &\ m_h(W'(u_h\t) - W'(u_h^*\t),\vphi_h^w) 
	\nonumber \\
	\label{eq:error equations - w}
	&\
	+ m_h(\vartheta_h,\vphi_h^w) - m_h(d_h^w\t,\vphi_h^w) ,
	\end{align}
\end{subequations}
with vanishing initial values $e_h^u(0) = 0$ and $e_h^w(0) = 0$. Since, $e_h^u(0)$ and $e_h^w(0)$ satisfies the error equation \eqref{eq:error equations - w} at $t = 0$, we obtain that
\begin{equation}
\label{eq:vartheta and d_w equivalence}
	\vartheta_h = d_h^w(0) .
\end{equation}  

The defects will be estimated using a discrete dual norm on the space $V_h$ defined by
\begin{equation}
\label{eq:star and H_h^-1 norm definition}
	\|d_h\|_{*,h} =  \sup_{0\neq v_h \in V_h} \frac{m(d_h,v_h^\ell)}{\|v_h\|_h}. 
\end{equation}
It is easy to see that, as in the continuous case, there exist constants $c,C > 0$ (independent of $h$) such that
\begin{align*}
	c \|v_h\|_{*,h} \leq |v_h|_h \leq C \|v_h\|_h .
\end{align*}

\subsection{Stability bounds}
\label{section:stability proof}

\begin{proposition}
	\label{proposition:stability}
	For sufficiently smooth solutions, e.g.~\eqref{eq:sufficient regularity}, and assuming that the defects satisfy
	\begin{align}
	\label{eq:defect bounds - assumed}
	\|d\t\|_{\ast,h} \leq &\ c h^2 \quadfor d = d_h^u, \dot d_h^u, d_h^w, \dot d_h^w, \ \text{ and for } \, 0 \leq t \leq T .
	\end{align}
	
	Then there exists $h_0 > 0$ such that the following stability estimate holds for all $h \leq h_0$ and $0 \leq t \leq T$:
	\begin{equation}
	\label{eq:stability bound}
	\begin{aligned}
	&\ \|e_h^u\t\|_h^2 + \|e_h^w\t\|_h^2 + \int_0^t \|\dot e_h^u(s)\|_h^2 \d s + \int_0^t \|e_h^w(s)\|_h^2 \d s \\
	\leq &\ C \Big( \|d_h^u(0)\|_{*,h}^2 + \|d_h^u\t\|_{*,h}^2 + t \|d_h^w(0)\|_{*,h}^2
	\\ &\ \phantom{\ C \Big( }
	+ \int_0^t \!\!\! \big( \|d_h^u(s)\|_{*,h}^2 + \|\dot d_h^u(s)\|_{*,h}^2 + \|d_h^w(s)\|_{*,h}^2 + \|\dot d_h^w(s)\|_{*,h}^2 \big) \d s \Big) ,
	\end{aligned}
	\end{equation}
	where the constant $C > 0$ is independent of $h$ and $t$, but depends \bbk exponentially \ebk on the final time $T$.
\end{proposition}

The main idea of the stability proof is to exploit the anti-symmetric structure of the semi-discrete error equations \eqref{eq:error equations} and combine it with multiple energy estimates, and which was originally developed for Willmore flow \cite{Willmore}.
The key issue in the stability proof is to establish a uniform-in-time $L^\infty(\Om)$ norm bound for $e_h^u$, which is then used to estimate the non-linear terms. Such time-uniform $L^\infty(\Om)$ bounds are obtained from the time-uniform $H^1(\Om)$ norm error estimates via an inverse estimate. 
The stability proof is completely independent of any geometric approximation errors, which only enter the consistency analysis.

Since the stability proof is entirely performed in the semi-discrete abstract setting of Section~\ref{subsection:semi discrete abstract CH equations}, we strongly expect that it can be generalised to other problems which can be cast in the same setting, for further details see Section~\ref{section:generalisations}.

In Section~\ref{section:consistency} we will show that the assumed bounds \eqref{eq:defect bounds - assumed} indeed hold. Hence, together with the stability bound \eqref{eq:stability bound}, the consistency bounds imply the estimates for the errors $e_h^u$ and $e_h^w$, for $0 \leq t \leq T$,
\begin{align*}
\|e_h^u\t\|_h \leq &\ C h^2 \andquad \|e_h^w\t\|_h \leq  C h^2 .
\end{align*}

\begin{proof}	
	In order to achieve the uniform-in-time stability bound, two sets of energy estimates are needed. These energy estimates strongly exploit the anti-symmetric structure of \eqref{eq:error equations}. (i) In the first, a uniform-in-time energy estimate is proved for $e_h^u$, but which comes with a critical term involving $\dot e_h^u$. (ii) The second estimate uses the time derivative of \eqref{eq:error equations - w}, and leads to a bound of this critical term and also to a uniform-in-time bound for $e_h^w$. The combination of these two energy estimates will give the stated stability bound.
	The structure and basic idea of the proof is sketched in Figure~\ref{fig:energy estimates}.
	\begin{figure}[htbp]
		\begin{center}
			\includegraphics[width=\textwidth]{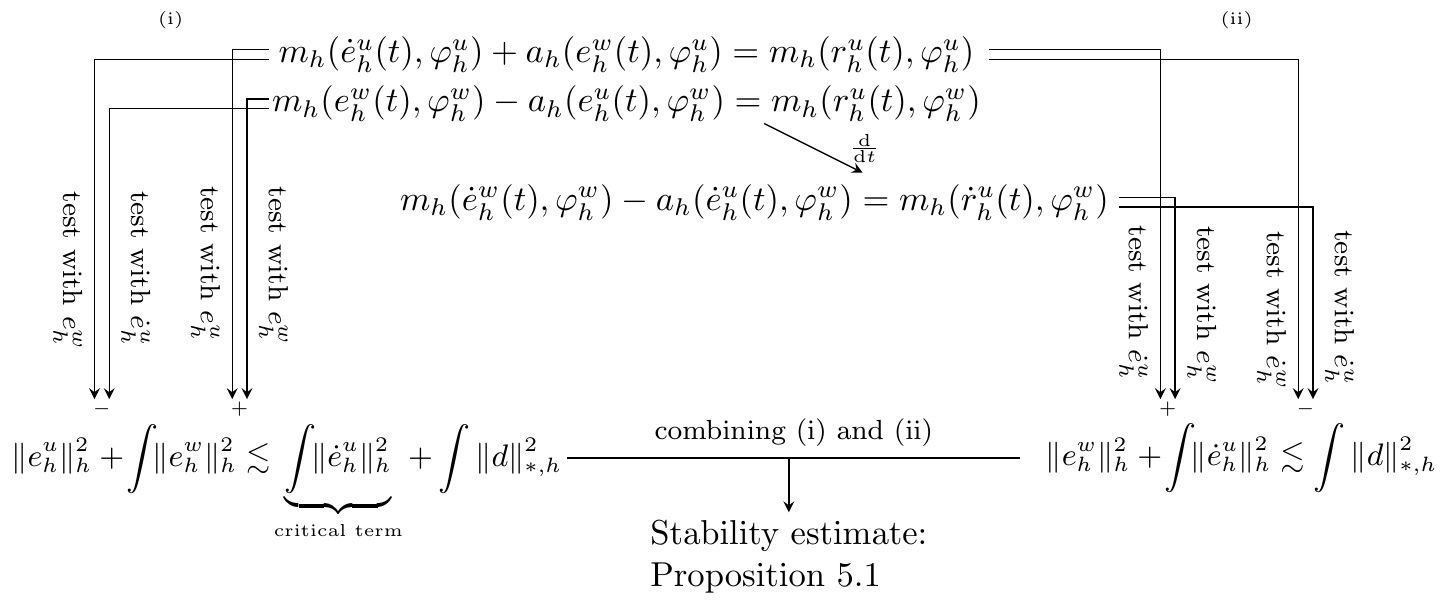}
			\caption{Sketch of the structure of the energy estimates of the stability proof. In the diagram $r_h^u$ and $r_h^w$ denote the right-hand sides of \eqref{eq:error equations - u} and \eqref{eq:error equations - w}.
			}
			\label{fig:energy estimates}
		\end{center}
	\end{figure}
	
	In order to handle the semi-linear term we first prove the stability bound on a time interval where the $L^\infty$ norm of $e_h^u$ is small enough, and then show that this time interval can be enlarged up to $T$.
	
	In the following $c$ and $C$ are generic constants that may take different values on different occurrences. Whenever it is possible, without confusion, we omit the argument $t$. By $\rho > 0$ we will denote a small number \bbk (chosen independently of $h$), \ebk used in Young's inequality, and hence we will often incorporate $h$ independent multiplicative constants into the those yet unchosen factors.

	\smallskip
	Let $t^* \in (0,T]$ be the maximal time such that the following estimate holds
	\begin{equation}
	\label{eq:L^infty bound}
	\|\ga_h e_h^u\t\|_{L^\infty(\Ga_h)} \leq \|e_h^u\t\|_{L^\infty(\Om_h)} \leq h^{1-d/4} \quadfor t \in [0,t^*] .
	\end{equation}
	Recall, that here we consider domains of dimension $d = 2$ or $3$, and note that for finite element functions the first inequality holds in general. We note that such a positive $t^*$ exists, since the initial error in $u$ is identically zero $e_h^u(0) = 0$. 

	\textit{Energy estimate (i):}
	We test \eqref{eq:error equations - u} with $e_h^u$ and \eqref{eq:error equations - w} with $e_h^w$, 
	then sum up the two equations, and use the symmetry of $a_h\ccdot$, to obtain
	\begin{equation}
	\label{eq:energy est - i - 1}
	\begin{aligned}
	m_h(\dot e_h^u,e_h^u) + m_h(e_h^w,e_h^w) = &\ m_h(W'(u_h) - W'(u_h^*),e_h^w) 
	- m_h(d_h^u,e_h^u) - m_h(d_h^w,e_h^w) + m_h(\vartheta_h,e_h^w) .
	\end{aligned}
	\end{equation}
	Similarly, we test \eqref{eq:error equations - u} with $e_h^w$ and \eqref{eq:error equations - w} with $\dot e_h^u$
	now subtracting the two equations and using the symmetry of $m_h(\cdot,\cdot)$, we obtain
	\begin{equation}
	\label{eq:energy est - i - 2}
	\begin{aligned}
	a_h(e_h^u ,\dot e_h^u)+a_h(e_h^w,e_h^w) = &\ m_h(W'(u_h) - W'(u_h^*),\dot e_h^u)  
	- m_h(d_h^u,e_h^w) + m_h(d_h^w,\dot e_h^u) - m_h(\vartheta_h,\dot e_h^u) .
	\end{aligned}
	\end{equation}
	Taking the linear combination of the equations \eqref{eq:energy est - i - 1} and \eqref{eq:energy est - i - 2}, we then use that for the bilinear form $a^* = a + m$ the following holds
	\begin{align}
	\label{eq:diff quadratic form}
	a_h^\ast(\dot e_h^u,e_h^u) = \Half \diff \|e_h^u\|_h^2 ,
	\end{align}
	and hence we obtain
	\begin{equation}
	\label{eq:energy est - i - pre estimates}
	\begin{aligned}
	\Half \diff \|e_h^u\|_h^2 + \|e_h^w\|_h^2 = &\ m_h(W'(u_h) - W'(u_h^*),e_h^w)  
	+ m_h(W'(u_h) - W'(u_h^*),\dot e_h^u) \\
	&\ + m_h(d_h^u,e_h^u) + m_h(d_h^u,e_h^w)  
	+ m_h(d_h^w,\dot e_h^u) - m_h(d_h^w,e_h^w) \\
	&\  + m_h(\vartheta_h,e_h^w) -  m_h(\vartheta_h,\dot e_h^u) .
	\end{aligned}
	\end{equation}
	
	Now we estimate all the terms on the right-hand side separately.
	
	For the non-linear terms, using the bound \eqref{eq:L^infty bound} together with the local Lipschitz continuity of the functions $W_\Om'$ and $W'_\Ga$, we obtain 
	\begin{equation}
	\label{eq:nonlinear term estimates - i - 1}
	\begin{aligned}
	m_h(W'(u_h) - W'(u_h^*),e_h^w) \leq &\ c |e_h^u|_h |e_h^w|_h 
	\leq c |e_h^u|_h^2 + \frac18 |e_h^w|_h^2 .
	\end{aligned}
	\end{equation}
	Analogously, for the other non-linear term we obtain
	\begin{equation}
	\label{eq:nonlinear term estimates - i - 2}
	\begin{aligned}
	m_h(W'(u_h) - W'(u_h^*),\dot e_h^u) \leq &\ c |e_h^u|_h |\dot e_h^u|_h
	\leq c |e_h^u|_h^2 + \rho |\dot e_h^u|_h^2 .
	\end{aligned}
	\end{equation}
	In both cases for the last inequalities we have used Young's inequality (in the second case with a small, \bbk $h$-independent \ebk factor $\rho>0$ chosen later on).
	
	For the terms with defects, using the Cauchy--Schwarz and Young's inequalities (often with a small, \bbk $h$-independent \ebk factor $\rho>0$ chosen later on), we obtain
	\begin{equation}
	\label{eq:defect term estimates - i}
	\begin{aligned}
	&\ 
	m_h(d_h^u,e_h^u) + m_h(d_h^u,e_h^w) 
	+ m_h(d_h^w,\dot e_h^u) - m_h(d_h^w,e_h^w) \\
	\leq &\ 
	\|d_h^u\|_{*,h} \|e_h^u\|_h + \|d_h^u\|_{*,h} \|e_h^w\|_h 
	+ \|d_h^w\|_{*,h} \|\dot e_h^u\|_h + \|d_h^w\|_{*,h} \|e_h^w\|_h \\
	\leq &\ c \|e_h^u\|_h^2 + \rho \|\dot e_h^u\|_h^2 + \frac18 \|e_h^w\|^2  
	+ c \Big( \|d_h^u\|_{*,h}^2 + \|d_h^w\|_{*,h}^2 \Big) .
	\end{aligned}
	\end{equation}
	
	The terms involving the correction $\vartheta_h$ are estimated similarly, in combination with \eqref{eq:vartheta and d_w equivalence}, as
	\begin{equation}
	\label{eq:vartheta term estimates - i}
	\begin{aligned}
	m_h(\vartheta_h,e_h^w) - m_h(\vartheta_h,\dot e_h^u) 
	\leq &\ 
	\|\vartheta_h\|_{*,h} \|e_h^w\|_h + \|\vartheta_h\|_{*,h} \|\dot e_h^u\|_h \\
	\leq &\ \rho \|\dot e_h^u\|_h^2 + \frac18 \|e_h^w\|^2  
	+ c \|d_h^w(0)\|_{*,h}^2 .
	\end{aligned}
	\end{equation}

	Altogether, by plugging in the estimates \eqref{eq:nonlinear term estimates - i - 1}--\eqref{eq:vartheta term estimates - i} into \eqref{eq:energy est - i - pre estimates}, 
	and integrating from $0$ to $t \leq t^*$, (and multiplying by two), we obtain the first energy estimate
	\begin{equation}
	\label{eq:energy estimate - i}
	\begin{aligned}
	\|e_h^u\t\|_h^2 + \int_0^t \|e_h^w(s)\|_h^2 \d s 
	\leq &\ \rho \int_0^t\|\dot e_h^u(s)\|_h^2 \d s 
	+ c \int_0^t\|e_h^u(s)\|_h^2 \d s  
	\\ &\
	+ c \int_0^t \|d_h^u(s)\|_{*,h}^2 + \|d_h^w(s)\|_{*,h}^2 \d s + c t \|d_h^w(0)\|_{*,h}^2 .
	\end{aligned}
	\end{equation}
	Note that the time-independent correction term $\vartheta_h$ have disappeared.
	
	\textit{Energy estimate (ii):}
	To control the first term on the right-hand side of \eqref{eq:energy estimate - i} we will now derive an energy estimate, which includes this term on the left-hand side. To this end we first differentiate the second equation of \eqref{eq:error equations} with respect to time, \bbk recall that by Proposition~\ref{proposition:semi-discrete well-posedness} $w_h$ is $C^1$ in time, \ebk and obtain the following system:
	\begin{subequations}
		\label{eq:error equations - dt}
		\begin{align}
		\label{eq:error equations - dt - u}
		m_h(\dot e_h^u, \vphi^u)+a_h(e_h^w,\vphi^u) = &\ - m_h(d_h^u,\vphi^u) \\
		\label{eq:error equations - dt - w}
		m_h(\dot e_h^w,\vphi^w)-a_h(\dot e_h^u , \vphi^w) = &\ m_h \Big(\diff \big( W'(u_h\t) - W'(u_h^*\t) \big) ,\vphi_h^w \Big)  
		- m_h(\dot d_h^w,\vphi^w).
		\end{align}	
	\end{subequations}
	
	Testing the error equation system \eqref{eq:error equations - dt} in two sets as before, and then taking directly the linear combination of the obtained equations would not lead to a feasible energy estimate, due to a critical term involving $\dot e_h^u$. Therefore, we first estimate the two equations separately, integrate in time (in (a) and (b), respectively), and then take their \emph{weighted} combination (in (c)).

	(a) We test \eqref{eq:error equations - dt - u} with $\dot e_h^u$ and \eqref{eq:error equations - dt - w} with $e_h^w$, then sum up the two equations, and use the symmetry of $a_h\ccdot$, to obtain
	\begin{equation}
	\label{eq:energy est - ii - 1}
	\begin{aligned}
	m_h(\dot e_h^u, \dot e_h^u) + m_h(\dot e_h^w,e_h^w) = &\ m_h \Big(\diff \big( W'(u_h\t) - W'(u_h^*\t) \big) ,e_h^w \Big) \\
	&\ - m_h(d_h^u,\dot e_h^u) - m_h(\dot d_h^w,e_h^w) .
	\end{aligned}
	\end{equation}
	
	For the non-linear terms, using the bound \eqref{eq:L^infty bound} together with the local Lipschitz continuity of the functions $W_\Om',W'_\Ga$ and $W_\Om'',W''_\Ga$, 
	and using that $u_h^\theta = u_h^* + \theta (u_h - u_h^*) = u_h^* + \theta e_h^u$, hence $\d / \d \theta \ u_h^\theta = e_h^u$ and $\dot u_h^\theta = \theta \dot e_h^u$, then the by the calculation
	\begin{align*}
	m_h \Big(\diff \big( W'(u_h\t) - W'(u_h^*\t) \big) ,e_h^w \Big)
	= &\ m_h\Big(\diff W'(u_h),e_h^w\Big) - m_h\Big(\diff W'(u_h^*),e_h^w\Big) \\
	= &\ \int_0^1 \frac{\d}{\d \theta } m_h\Big(\diff W'(u_h^\theta),e_h^w\Big) \d \theta \\
	= &\ \int_0^1 \frac{\d}{\d \theta } m_h\Big(W''(u_h^\theta) \dot u_h^\theta,e_h^w\Big) \d \theta \\
	= &\ \int_0^1 m_h\big( W'''(u_h^\theta) e_h^u + W''(u_h^\theta) \dot e_h^u , e_h^w\big) \d \theta \\
	\leq &\ c \big( |e_h^u|_h + |\dot e_h^u| \big) |e_h^w|_h .
	\end{align*}
	Therefore, for the non-linear term in \eqref{eq:energy est - ii - 1} we obtain 
	\begin{equation}
	\label{eq:nonlinear term estimates - ii - 1}
	\begin{aligned}
	m_h \Big(\diff \big( W'(u_h\t) - W'(u_h^*\t) \big) ,e_h^w \Big) \leq &\ c \big( |e_h^u|_h + |\dot e_h^u|_h \big) |e_h^w|_h \\
	\leq &\ \rho |\dot e_h^u|_h^2 + c |e_h^u|_h^2 + c |e_h^w|_h^2 .
	\end{aligned}
	\end{equation}
	
	For the terms with defects, using the Cauchy--Schwarz and Young's inequalities (often with a small, \bbk $h$-independent \ebk factor $\rho>0$ chosen later on), we obtain
	\begin{equation}
	\label{eq:defect term estimates - ii - 1}
	\begin{aligned}
	- m_h(d_h^u,\dot e_h^u) - m_h(\dot d_h^w,e_h^w) \leq &\ 
	\|d_h^u\|_{*,h} \|\dot e_h^u\|_h + \|\dot d_h^w\|_{*,h} \|e_h^w\|_h \\
	\leq &\ \rho \|\dot e_h^u\|_h^2 + c \|e_h^w\|_h^2 + c \big( \|d_h^u\|_{*,h}^2 + \|\dot d_h^w\|_{*,h}^2 \big) .
	\end{aligned}
	\end{equation}
	
	Combining the above estimates for \eqref{eq:energy est - ii - 1}, using \eqref{eq:diff quadratic form}, integrating in time from $0$ to $t \leq t^*$, and multiplying by two, we altogether obtain
	\begin{equation}
	\label{eq:energy est - ii - 1 - post int}
	\begin{aligned}
	\int_0^t |\dot e_h^u(s)|_h^2 \d s + |e_h^w\t|_h^2
	\leq &\ \rho \int_0^t \|\dot e_h^u(s)\|_h^2 \d s 
	+ c \int_0^t \big( \|e_h^u(s)\|_h^2 + \|e_h^w(s)\|_h^2 \big) \d s \\
	&\ + c\int_0^t  \big( \|d_h^u(s)\|_{*,h}^2 + \|\dot d_h^w(s)\|_{*,h}^2 \big) \d s .
	\end{aligned}
	\end{equation}

	(b) We now test \eqref{eq:error equations - dt - u} with $\dot e_h^w$ and \eqref{eq:error equations - dt - w} with $\dot e_h^u$
	now subtracting the two equations and using the symmetry of $m_h(\cdot,\cdot)$, we obtain
	\begin{equation*}
	\begin{aligned}
	a_h(e_h^w,\dot e_h^w)+a_h(\dot e_h^u , \dot e_h^u) = &\ - m_h \Big(\diff \big( W'(u_h\t) - W'(u_h^*\t) \big) ,\dot e_h^u \Big) \\
	&\ - m_h(d_h^u,\dot e_h^w) + m_h(\dot d_h^w,\dot e_h^u) .
	\end{aligned}
	\end{equation*}
	The term $\dot e_h^w$ cannot be absorbed or controlled, therefore, by the product rule we rewrite it as follows:
	\begin{equation*}
	- m_h(d_h^u,\dot e_h^w) = - \diff m_h(d_h^u,e_h^w) + m_h(\dot d_h^u,e_h^w) .
	\end{equation*}
	The combination of the above two equations then gives
	\begin{equation}
	\label{eq:energy est - ii - 2}
	\begin{aligned}
	a_h(e_h^w,\dot e_h^w)+a_h(\dot e_h^u , \dot e_h^u) = &\ - m_h \Big(\diff \big( W'(u_h\t) - W'(u_h^*\t) \big) ,\dot e_h^u \Big) \\
	&\ - \diff m_h(d_h^u,e_h^w) + m_h(\dot d_h^u,e_h^w) + m_h(\dot d_h^w,\dot e_h^u) .
	\end{aligned}
	\end{equation}
	
	For the non-linear term in \eqref{eq:energy est - ii - 2}, analogously as before, we obtain
	\begin{equation}
	\label{eq:nonlinear term estimates - ii - 2}
	\begin{aligned}
	m_h \Big(\diff \big( W'(u_h\t) - W'(u_h^*\t) \big) , \dot e_h^u \Big) \leq &\ c \big( |e_h^u|_h + |\dot e_h^u| \big) |\dot e_h^u|_h \\
	\leq &\ c_0 |\dot e_h^u|_h^2 + c |e_h^u|_h^2 ,
	\end{aligned}
	\end{equation}
	with a particular ($h$ independent) constant $c_0 > 0$.
	
	The defect terms without a time derivative are bounded, similarly as before, by 
	\begin{equation}
	\label{eq:defect term estimates - ii - 2}
	\begin{aligned}
	m_h(\dot d_h^u,e_h^w) + m_h(\dot d_h^w,\dot e_h^u) \leq &\ 
	\|\dot d_h^u\|_{*,h} \|e_h^w\|_h + \|\dot d_h^w\|_{*,h} \|\dot e_h^u\|_h \\
	\leq &\ \rho \|\dot e_h^u\|_h^2 + c \|e_h^w\|^2 + c \big( \|\dot d_h^u\|_{*,h}^2 + \|\dot d_h^w\|_{*,h}^2 \big) .
	\end{aligned}
	\end{equation}
	
	
	Combining the above estimates for \eqref{eq:energy est - ii - 2}, using \eqref{eq:diff quadratic form} for the bilinear form $a_h\ccdot$, integrating in time from $0$ to $t \leq t^*$, and multiplying by two, we altogether obtain
	\begin{equation*}
	\begin{aligned}
	\int_0^t \|\dot e_h^u(s)\|_{a_h}^2 \d s + \half \|e_h^w\t\|_{a_h}^2 
	\leq &\ \half \|e_h^w(0)\|_{a_h}^2 
	+ c_0 \int_0^t |\dot e_h^u(s)|_h^2 \d s + \rho \int_0^t \|\dot e_h^u(s)\|_h^2 \d s\\
	&\ + c \int_0^t \big( \|e_h^u(s)\|^2 + \|e_h^w(s)\|^2 \big) \d s \\
	&\ - m_h(d_h^u\t,e_h^w\t)
	+ c \int_0^t \big( \|\dot d_h^u(s)\|_{*,h}^2 + \|\dot d_h^w(s)\|_{*,h}^2 \big) \d s .
	\end{aligned}
	\end{equation*}
	
	Finally, by estimating the pointwise defect terms similarly as, e.g., in \eqref{eq:defect term estimates - ii - 2}, we obtain
	\begin{equation}
	\label{eq:energy est - ii - 2 - post int}
	\begin{aligned}
	\int_0^t \|\dot e_h^u(s)\|_{a_h}^2 \d s + \half \|e_h^w\t\|_{a_h}^2 
	\leq &\ c_0 \int_0^t |\dot e_h^u(s)|_h^2 \d s + \rho \int_0^t \|\dot e_h^u(s)\|_h^2 \d s\\
	&\ + c \int_0^t \big( \|e_h^u(s)\|^2 + \|e_h^w(s)\|^2 \big) \d s \\
	&\ + \rho \|e_h^w\t\|_h^2 + c \|d_h^u\t\|_{*,h}^2 \\
	&\ + c \int_0^t \big( \|\dot d_h^u(s)\|_{*,h}^2 + \|\dot d_h^w(s)\|_{*,h}^2 \big) \d s .
	\end{aligned}
	\end{equation}
	
	(c) We now multiply \eqref{eq:energy est - ii - 1 - post int} by $2 c_0$ and \eqref{eq:energy est - ii - 2 - post int} by $1$, and take this weighted combination of the two inequalities. After collecting the terms, we obtain
	\begin{equation}
	\label{eq:energy est - ii - post sum}
	\begin{aligned}
	&\ 2 c_0 \int_0^t |\dot e_h^u(s)|_h^2 \d s + \int_0^t \|\dot e_h^u(s)\|_{a_h}^2 \d s 
	+ 2 c_0 |e_h^w\t|_h^2 + \half \|e_h^w\t\|_{a_h}^2 \\
	\leq &\  c_0 \int_0^t |\dot e_h^u(s)|_h^2 \d s + \rho \int_0^t \|\dot e_h^u(s)\|_h^2 \d s  
	+ \rho \|e_h^w\t\|_h^2 \\
	&\ + c \int_0^t \big( \|e_h^u(s)\|_h^2 + \|e_h^w(s)\|_h^2 \big) \d s \\
	&\ + c \big( \|d_h^u\t\|_{*,h}^2 + \|d_h^u(0)\|_{*,h}^2 \big) 
	+ c \int_0^t  \big( \|d_h^u(s)\|_{*,h}^2 + \|\dot d_h^u(s)\|_{*,h}^2 + \|\dot d_h^w(s)\|_{*,h}^2 \big) \d s . 
	\end{aligned}
	\end{equation}

	The first term with $\dot e_h^u$ on the right-hand side is directly absorbed into the first term on the left-hand side. After combining the norms on the left-hand side, the second term on the right-hand side is absorbed by choosing $\rho > 0$ small enough. We use a further absorption into the term $e_h^w\t$, and then divide both sides by $\min\{1/4,c_0\}$, which yields the second energy estimate
	\begin{equation}
	\label{eq:energy estimate - ii}
	\begin{aligned}
	\int_0^t \|\dot e_h^u(s)\|_h^2 \d s + \|e_h^w\t\|_h^2
	\leq &\ c \int_0^t \big( \|e_h^u(s)\|_h^2 + \|e_h^w(s)\|_h^2 \big) \d s \\
	&\ + c \big( \|d_h^u\t\|_{*,h}^2 + \|d_h^u(0)\|_{*,h}^2 \big) \\
	&\ + c \int_0^t \!\!\! \big( \|d_h^u(s)\|_{*,h}^2 + \|\dot d_h^u(s)\|_{*,h}^2 + \|\dot d_h^w(s)\|_{*,h}^2 \big) \d s . 
	\end{aligned}
	\end{equation}

	\textit{Combining the energy estimates:}
	The sum of the energy estimates \eqref{eq:energy estimate - i} and \eqref{eq:energy estimate - ii} gives  
	\begin{equation}
	\label{eq:energy estimate - pre Gronwall}
	\begin{aligned}
	&\ \|e_h^u\t\|_h^2 + \|e_h^w\t\|_h^2 + \int_0^t \|\dot e_h^u(s)\|_h^2 \d s + \int_0^t \|e_h^w(s)\|_h^2 \d s \\
	\leq &\ \rho \int_0^t\|\dot e_h^u(s)\|_h^2 \d s 
	+ c \int_0^t \big( \|e_h^u(s)\|_h^2 + \|e_h^w(s)\|_h^2 \big) \d s  + c t \|d_h^w(0)\|_{*,h}^2 \\
	&\ + c \big( \|d_h^u\t\|_{*,h}^2 + \|d_h^u(0)\|_{*,h}^2 \big) 
	+ c \int_0^t \!\!\! \big( \|d_h^u(s)\|_{*,h}^2 + \|\dot d_h^u(s)\|_{*,h}^2 + \|d_h^w(s)\|_{*,h}^2 + \|\dot d_h^w(s)\|_{*,h}^2 \big) \d s .
	\end{aligned}
	\end{equation}
	A final absorption, and Gronwall's inequality gives the stated stability estimate in $[0,t^*]$.
	
	It is left to show that in fact $t^*$ coincides with $T$ for sufficiently small $h \leq h_0$. We prove this by the following argument: By the proven stability bound and the assumed bounds \eqref{eq:defect bounds - assumed} the error $e_h^u\t$ in $[0,t^*]$ satisfies
	\begin{equation*}
	\|e_h^u\|_h \leq C h^2.
	\end{equation*}
	Then, by the inverse estimate, \cite[Theorem~4.5.11]{BrannerScott}, we have, for $t \in [0,t^*]$,
	\begin{equation}
	\label{eq:Linfty bound via inverse estimate}
	\begin{aligned}
	\|e_h^u\|_{L^\infty(\Om_h)} \leq &\ c h^{-d/2} \|e_h^u\|_{L^2(\Om_h)} \\
	\leq &\ c h^{-d/2} \|e_h^u\|_h \\
	\leq &\ c C h^{-d/2} h^2 \leq \half h^{1-d/4} ,
	\end{aligned}
	\end{equation}
	for sufficiently small $h$. Therefore, we can extend the bounds \eqref{eq:L^infty bound} beyond $t^*$, which contradicts the maximality of $t^*$ unless $t^* = T$. Hence, we have the stability bound \eqref{eq:stability bound} for $t\in[0,T]$.
	\end{proof}

\subsection{Generalisations}
\label{section:generalisations}

In this section we illustrate the versatility of the above stability analysis of Section~\ref{section:stability proof}. \bbk We use the generalised problem below to demonstrate the generality of the energy estimates we used in the stability proof, Proposition~\ref{proposition:stability}, detailing how such a more general problem would be handled (if well-posedness is provided). We also remark here that the employed abstract framework can be modified to a different setting, i.e.~the particular choice of the spaces $V$ and $H$, or the bilinear forms $a$ and $m$ are inessential. \ebk 

First, let us consider the standard Cahn--Hilliard equation \cite{CahnHilliard} in a smooth domain with homogeneous Neumann boundary conditions. Written as a second order system it reads: The functions $u,w: \Om \times [0,T] \to \R$ satisfying the PDEs \eqref{eq:CH system - u Om}--\eqref{eq:CH system - w Om} in $\Om$, which are endowed with the boundary conditions
\begin{equation}
\label{eq:CH homogeneous Neumann bc}
	\partial_\nu u =  \partial_\nu w = 0 \quad \text{on } \Ga .
\end{equation}
This problem can be cast in the abstract setting of Section~\ref{section:abstract formulation}, by setting $V = \{ v \in H^1(\Om) \mid \int_\Om u = 0 \}$ and $H = L^2(\Om)$, with the bilinear forms on these spaces:
\begin{equation*}
	a(u,v) = \int_\Om \nb u \cdot \nb v , \andquad
	m(u,v) = \int_\Om u v .
\end{equation*}
Therefore, the weak formulation to the Cahn--Hilliard equation with boundary conditions \eqref{eq:CH homogeneous Neumann bc} reads exactly as \eqref{eq:abstract weak form}.
Hence, with the analogous modifications, the bulk--surface finite element semi-discreti\-satsion can be written exactly as \eqref{eq:semi-discrete problem}. The proof of Proposition~\ref{proposition:stability} immediately holds.

A similar setting can be used for inhomogeneous Neumann, or Dirichlet, or mixed boundary conditions.

Let us now consider a \emph{generalised} semi-linear \bbk Cahn--Hilliard-type \ebk equation with Cahn--Hilliard-type dynamic boundary conditions, written in the abstract setting of Section~\ref{section:abstract formulation}: Find $u\in C^1([0,T],H) \cap L^2([0,T],V)$ and $w \in L^2([0,T],V)$ such that, for time $0 < t\leq T$ and all $\vphi^u, \vphi^w \in V$,
\begin{subequations}
	\label{eq:CH general nonlin}
	\begin{alignat}{2}
		\label{eq:CH general nonlin - u}
		m(\dot u\t,\vphi^u) + a(w\t,\vphi^u) = &\ m(f(u\t),\vphi^u) , \\
		m(w\t,\vphi^w) - a(u\t,\vphi^w) = &\ m(g(u\t),\vphi^w)  ,
	\end{alignat}
\end{subequations}
where the two nonlinear terms are defined similarly as \eqref{eq:nonlinear term def}.
Both nonlinearities may contain bulk and boundary terms, \bbk and assumed the same local Lipschitz conditions as $W$ before. \ebk The corresponding modified system (with $\vartheta_h$ as before) is analogous.

\bbk An example which would fit into this framework are the Cahn--Hilliard equation (still endowed with dynamic boundary conditions) with a proliferation term \cite[Section~3]{Miranville}, or \cite[Chapter~8]{Miranville_book}, i.e.~(in fourth-order strong form) $\dot u - \big( -\laplace u + g_\Om(u)\big) + f_\Om(u) = 0$ in $\Om$, and analogously on the boundary $\Ga$. The problem \eqref{eq:CH general nonlin} is just slightly less general than the fourth-order non-linear parabolic PDE in \cite{Dlotko} or \cite{Willmore}. Well-posedness results should be studied in the particular cases; for semi-discretisations, e.g., analogously to Proposition~\ref{proposition:semi-discrete well-posedness} (which directly shows short-time existence and time regularity). \ebk 

To this problem the proof of Proposition~\ref{proposition:stability} holds \emph{only} subject to the following modification \bbk (assuming sufficient regularity conditions): \ebk 
\begin{itemize}
	\item[-] The estimates of the nonlinear terms in the first error equation (corresponding to \eqref{eq:CH general nonlin - u}), are estimated analogously to \eqref{eq:nonlinear term estimates - i - 1} and \eqref{eq:nonlinear term estimates - i - 2}, etc.
\end{itemize}

The Cahn--Hilliard equation with Allen--Cahn-type dynamic boundary conditions \cite{CherfilsPetcuPierre_2010}, i.e.~the PDE system \eqref{eq:CH system - u Om}--\eqref{eq:CH system - w Om} in $\Om$ endowed with the dynamic boundary conditions
\begin{equation}
\label{eq:CH Allen-Cahn bc}
\begin{aligned}
\dot u = &\ \Delta_\Ga u - W'_\Ga(u) - \partial_\nu u , \\
\pa_\nu w = &\ 0 , 
\end{aligned}\qquad \text{on } \Ga ,
\end{equation}
unfortunately, does not fit into the abstract framework of this paper: compare the weak formulation (10)--(11) in \cite{CherfilsPetcuPierre_2010} with \eqref{eq:weak form}, or \eqref{eq:abstract weak form} above.
%

\section{Consistency}
\label{section:consistency}

The consistency analysis relies on  error estimates of the nodal interpolations in the bulk and on the surface, error estimates for the Ritz map, geometric approximation errors in the bilinear forms, and a technical result for estimating norms on a boundary layer. 

\subsection{Geometric errors}

Let us recall our assumptions on the bulk and the surface, and on their discrete counterparts: the bounded domain $\Om \subset\R^d$ ($d=2$ or $3$) has an (at least) $C^2$ boundary $\Gamma$; the quasi-uniform triangulation $\Om_h$ (approximating $\Om$) whose boundary $\Ga_h := \pa \Om_h$ is an interpolation of $\Ga$.

\subsubsection{Interpolation and Ritz map error estimates}

The piecewise linear finite element interpolation operator $\widetilde I_h v \in V_h$, with lift $I_h v = (\widetilde I_h v)^\ell \in V_h^\ell$ satisfies the following bounds.
\begin{lemma}
	\label{lemma:interpolation error est}
	For $v\in H^2(\Om)$, such that $\gamma v \in H^2(\Ga)$. The piecewise linear finite element interpolation satisfies the following estimates:
	\begin{enumerate}
		\item[(i)] Interpolation error in the \emph{bulk}; see \cite{Bernardi,ElliottRanner}:
		\begin{equation*}
		\|v - I_h v\|_{L^2(\Om)} + h \|\nb(v - I_h v)\|_{L^2(\Om)} \leq C h^2 \|v\|_{H^2(\Om)}.
		\end{equation*}
		\item[(ii)] Interpolation error on the \emph{surface}, for $d \leq 3$; see \cite{Dziuk88}:
		\begin{equation*}
		\|\gamma(v - I_h v)\|_{L^2(\Ga)} + h \|\nb_\Ga(v - I_h v)\|_{L^2(\Ga)}
		\leq Ch^2 \|\gamma v\|_{H^2(\Ga)}.
		\end{equation*}
	\end{enumerate}
\end{lemma}

From \cite[Lemma~3.11 and 3.15]{dynbc} (with $\beta = 1$ therein) we recall the following estimates for the error in the Ritz map. 
\newcommand{\vr}{v-R_hv}
\begin{lemma}
	\label{lemma:Ritz error est} 
	For any $v\in H^2(\Omega)$ with $\gamma v \in H^2(\Gamma)$ the error of the Ritz map \eqref{eq:Ritz map definition} satisfies the following bounds, for $h \leq h_0$,
	\begin{align*}
	\|\vr\|_{H^1(\Om)} +  \|\gamma(\vr)\|_{H^1(\Ga)} \leq &\ C h \, \big( \|v\|_{H^2(\Om)}  + \|\gamma v\|_{H^2(\Gamma)} \big) , \\
	\|\vr\|_{L^2(\Om)} +  \|\gamma(\vr)\|_{L^2(\Ga)} \leq &\ C h^2 \, \big( \|v\|_{H^2(\Om)}  + \|\gamma v\|_{H^2(\Gamma)} \big) , 
	\end{align*}
	where the constant $C$ is independent of $h$ and $v$.
\end{lemma}

\subsubsection{Geometric approximation errors}
\label{section:geometric approx errors}

The following technical result from {\cite[Lemma~6.3]{ElliottRanner}} helps to estimate norms on a layer of triangles around the boundary.
\begin{lemma}
	\label{lemma:layer est}
	For all $v\in H^1(\Om)$ the following estimate holds:
	\begin{equation}
	\|v\|_{L^2(B_h^\ell)} \leq C h^\Half \|v\|_{H^1(\Om)},
	\end{equation}
	where $B_h^\ell$ collects the lifts of elements which have at least two nodes on the boundary.
\end{lemma}

The bilinear forms $a$ and $a_h$, and $m$ and $m_h$, from \eqref{eq:bilinear forms} and \eqref{eq:discrete bilinear forms}, satisfy the following geometric approximation estimate, proved in {\cite[Lemma~3.9]{dynbc}}.
\begin{lemma}
	\label{lemma:geometric errors}
	The bilinear forms \eqref{eq:bilinear forms} and their discrete counterparts \eqref{eq:discrete bilinear forms} satisfy the following estimates for $h \leq h_0$, for any $v_h,w_h \in V_h$, 
	\begin{align*}
	|a(v_h^\ell,w_h^\ell) - a_h(v_h,w_h)| \leq &\ Ch \|\nb v_h^\ell\|_{L^2(B_h^\ell)} \, \|\nb w_h^\ell\|_{L^2(B_h^\ell)} \\ 
	&\ + C h^2 \bigg(\|\nb v_h^\ell\|_{L^2(\Om)} \, \|\nb w_h^\ell\|_{L^2(\Om)} + \|\nb_\Ga v_h^\ell\|_{L^2(\Ga)} \, \|\nb_\Ga w_h^\ell\|_{L^2(\Ga)}\bigg),
	\\
	|m(v_h^\ell,w_h^\ell) - m_h(v_h,w_h)| \leq &\ Ch \|v_h^\ell\|_{L^2(B_h^\ell)} \, \|w_h^\ell\|_{L^2(B_h^\ell)} \\
	&\ + C h^2 \Big(\|v_h^\ell\|_{L^2(\Om)} \, \|w_h^\ell\|_{L^2(\Om)} + \|\gamma v_h^\ell\|_{L^2(\Ga)} \, \|\gamma w_h^\ell\|_{L^2(\Ga)}\Big) .
	\end{align*}
\end{lemma}

The combination of the two estimates of Lemma~\ref{lemma:geometric errors} yields a similar estimate between the bilinear forms $a^*$ and $a_h^*$.

As a consequence we also have the $h$-uniform equivalence of the norms $\|\cdot\|$ and $\|\cdot\|_h$ induced by the bilinear forms $a^*$ and $a_h^*$, respectively, of the norms $|\cdot|$ and $|\cdot|_h$ induced by the bilinear forms $m$ and $m_h$, respectively:
\begin{equation}
\label{eq:norm equivalence}
\|v_h^\ell\| \sim \|v_h\|_h \quad\hbox{and}\quad |v_h^\ell | \sim |v_h |_h
\quad\hbox{uniformly in $h$.} 
\end{equation}

\subsection{Defect bounds}

In this section we prove bounds for \bbk the bulk--surface finite element residuals \ebk and for their time derivatives, i.e.~we prove that condition \eqref{eq:defect bounds - assumed} of Proposition~\ref{proposition:stability} is satisfied.

\begin{proposition}
	\label{proposition:defect bounds}
	Let $(u,w)$ be a solution of \eqref{eq:CH system} that satisfies the regularity conditions \eqref{eq:sufficient regularity}. 	
	Then the defects $d_h^u(\cdot,t)$ and $d_h^w(\cdot,t) \in V_h$ from \eqref{eq:defect definition} and their time derivatives satisfy the bounds, for $0 \leq t \leq T$,
	\begin{align}
	\|d_h^u(\cdot,t)\|_{\ast,h} \leq &\ C h^2 \andquad \|\dot d_h^u(\cdot,t)\|_{\ast,h} \leq C h^2, \\
	\|d_h^w(\cdot,t)\|_{\ast,h} \leq &\ C h^2 \andquad \|\dot d_h^w(\cdot,t)\|_{\ast,h} \leq C h^2,
	\end{align}
	where the constant $C > 0$ depends on the final time $T$, on the Sobolev norms of the solution, but it is independent from $h$ and $t$.
\end{proposition}
\begin{proof}
	The proof is in the standard spirit of consistency estimates comparing the exact solutions to their Ritz maps, and uses geometric approximation errors from above. For such proofs see, e.g., \cite{DziukElliott_L2} for linear evolving surface PDEs, and \cite{dynbc,Hip17,dynbc_waves_L2} for problems with dynamic boundary conditions. Again, within the proof we omit the time dependencies. For more details to the present proof we refer to \cite{Harder_thesis}.

	We will first prove the consistency bounds for the defect in $u$ and its time derivative, and then, by using analogous techniques, we will show the same estimates for the defect in $w$ and its time derivative. 
	
	\emph{Bounds for $d_h^u$ and $\dot d_h^u$:}
	To estimate the defects in $u$, we start by subtracting \eqref{eq:abstract weak form - u} from \eqref{eq:defect definition - u} with $\vphi_h^\ell \in V_h^\ell$ and $\vphi_h \in V_h$, respectively, as test functions, then use the definition of the Ritz map \eqref{eq:Ritz map definition} to obtain
	\begin{equation}
	\label{eq:defect u - pre estimates}
	\begin{aligned}
	m_h(d_h^u,\phi_h^u) 
	= &\ m_h(\widetilde{R}_h \dot u, \vphi_h) - m(\dot u,\vphi_h^\ell) 
	\\	&\ 
	+ a_h(\widetilde{R}_h w,\vphi_h) - a(w,\vphi_h^\ell) \\
	= &\ \big( m_h(\widetilde{R}_h \dot u, \vphi_h) - m(R_h \dot u,\vphi_h^\ell) \big) + m(R_h \dot u - \dot u,\vphi_h^\ell) \\
	&\ - \big( m_h(\widetilde{R}_h w,\vphi_h) - m(R_h w,\vphi_h^\ell) \big) - m(R_h w - w,\vphi_h^\ell) ,
	\end{aligned}
	\end{equation}
	where for the last equality we added and subtracted the appropriate intermediate terms.
	The terms in the last two lines are estimated separately: the first terms by the geometric approximation estimates of Lemma~\ref{lemma:geometric errors} together with Lemma~\ref{lemma:layer est}, the second terms by a Cauchy--Schwarz inequality and then by the Ritz map error bounds from Lemma~\ref{lemma:Ritz error est}. Altogether we obtain
	\begin{align*}
	m_h(d_h^u,\vphi_h) \leq c h^2 \|R_h \dot u\| \, \|\vphi_h^\ell\| + c h^2 (\|\dot u\|_{H^2(\Om)} + \|\ga \dot u\|_{H^2(\Ga)}) \, |\vphi_h^\ell|.
	\end{align*}
	By using the $\|\cdot\|$ norm error estimates for the Ritz map within the first term here, we obtain:
	\begin{align*}
	\|R_h \dot u\| \leq \|R_h \dot u - \dot u\| + \|\dot u\| \leq (1 + c h) (\|\dot u\|_{H^2(\Om)} + \|\ga \dot u\|_{H^2(\Ga)}) .
	\end{align*}
	Altogether, after recalling the definition of the discrete dual norm \eqref{eq:star and H_h^-1 norm definition}, we obtain that the defect in $u$ is bounded by
	\begin{align*}
	\|d_h^u\|_{\ast,h} \leq ch^2 (\|\dot u\|_{H^2(\Om)} + \|\ga \dot u\|_{H^2(\Ga)}) .
	\end{align*}
	
	The time derivative of the defect $d_h^u$ is bounded by the same techniques. We take the time derivative of the equation \eqref{eq:defect u - pre estimates}, use that $\vphi_h \in V_h$ is time independent and that $\d / \d t \, (R_h \dot u ) = R_h \ddot u$, and then use the same techniques as above to obtain
	\begin{align*}
	\|\dot d_h^u\|_{\ast,h} \leq c h^2 (\|\ddot u\|_{H^2(\Om)} + \|\ga \ddot u\|_{H^2(\Ga)}) .
	\end{align*}
	
	\emph{Bounds for $d_h^w$ and $\dot d_h^w$:} 
	To estimate the defects in $w$ we use the same approach as for $u$. We start by subtracting \eqref{eq:abstract weak form - w} from \eqref{eq:defect definition - w} with $\vphi_h^\ell \in V_h^\ell$ and $\vphi_h \in V_h$, respectively, as test functions. Using again the definition of the Ritz map \eqref{eq:Ritz map definition} and we again add terms to gain the structure as before:
	\begin{equation}
	\label{eq:defect w - pre estimates}
	\begin{aligned}
	m_h(d_h^u,\phi_h^u) 
	= &\ \big( m_h(\widetilde{R}_h w,\vphi_h)  - m(R_h w,\vphi_h^\ell) \big) + m(R_h w - w,\vphi_h^\ell)  \\
	&\ + \big( m_h(\widetilde{R}_h u,\vphi_h) - m(R_h u,\vphi_h^\ell) \big) + m(R_h u - u,\vphi_h^\ell) \\
	&\ - \big( m_h(W'(\widetilde{R}_h u),\vphi_h) - m((W'(\widetilde{R}_h u))^\ell,\vphi_h^\ell) \big)
	- m(W'(R_h u) - W'(u),\vphi_h^\ell) ,
	\end{aligned}
	\end{equation}
	where for the non-linear term we have used the fact that, for an arbitrary function $f$ and for any $ v_h \in V_h$, there holds $f(v_h^\ell)=f(v_h\circ G_h^{-1})=(f\circ v_h)\circ G_h^{-1}=(f(v_h))^\ell$, with $G_h$ defined in \eqref{eq:bulk mapping}.
	
	The terms in the first two lines on the right-hand side are estimated exactly as before for $d_h^u$, as 
	\begin{align*}
	ch^2 (\|u\|_{H^2(\Om)} + \|\ga u\|_{H^2(\Ga)} + \|w\|_{H^2(\Om)} + \|\ga w\|_{H^2(\Ga)} ) .
	\end{align*}
	The remaining non-linear terms are bounded similarly as before. The first term by the geometric approximation estimates Lemma~\ref{lemma:geometric errors} together with Lemma~\ref{lemma:layer est} as
	\begin{align*}
	m_h(W'(\widetilde{R}_h u),\vphi_h) - m((W'(\widetilde{R}_h u))^\ell,\vphi_h^\ell) 
	\leq &\ c h^2 \|W'(R_h u)\| \|\vphi_h^\ell\| .
	\end{align*}
	The second term is estimated by a Cauchy--Schwarz inequality and then by the Ritz map error bounds from Lemma~\ref{lemma:Ritz error est}.
	\begin{align*}
	m(W'(R_h u) - W'(u),\vphi_h^\ell) \leq &\ |W'(R_h u) - W'(u)| \, |\vphi_h^\ell| .
	\end{align*}
	
	It is left to estimate the non-linear terms involving the Ritz map. We first establish a bound for $\|W'(R_h u)\|$, and start by decomposing this norm into its bulk and surface parts
	\begin{align*}
	\|W'(R_h u)\| \leq &\ \|\nb (W_\Om'(R_h u))\|_{L^2(\Om)} + \|W_\Om'(R_h u)\|_{L^2(\Om)} \\
	&\ + \|\nbg (W_\Ga'(\ga(R_h u)))\|_{L^2(\Om)} + \|W_\Ga'(\ga(R_h u))\|_{L^2(\Ga)} .
	\end{align*}
	The above terms on the right-had side are estimated separately, but by analogous techniques. For the first term we have
	\begin{align*}
	\|\nb (W_\Om'(R_h u))\|_{L^2(\Om)} 
	\leq &\ 
	\|W_\Om''(R_h u)\|_{L^\infty(\Om)} \|\nb R_h u\|_{L^2(\Om)} \\
	\leq &\ \|W_\Om''(R_h u)\|_{L^\infty(\Om)} \big( \|\nb R_h u - u\|_{L^2(\Om)} + \|u\|_{L^2(\Om)} \big) \\
	\leq &\ \|W_\Om''(R_h u)\|_{L^\infty(\Om)} (1 + c h^2) \|u\|_{H^2(\Om)} ,
	\end{align*}
	which is bounded by a constant, provided an $L^\infty(\Om)$ norm bound on $R_h u$. Under the same condition, the other three terms are also bounded independently of $h$, by a similar argument.
	
	To show a bound for $\|R_h u\|_{L^\infty(\Om)}$, we use an inverse estimate \cite[Theorem~4.5.11]{BrannerScott} (recall that $d=2$ or $3$) and the $L^\infty(\Om)$-stability of the finite element interpolation, which yield
	\begin{equation}
	\label{eq:Ritz map Linfty bound}
	\begin{aligned}
	\|R_h u\|_{L^\infty(\Om)} \leq &\ \|R_h u - I_h u\|_{L^\infty(\Om)} + \|I_h u\|_{L^\infty(\Om)} \\
	\leq &\ c h^{-d/2} \|R_h u - I_h u\|_{L^2(\Om)} + \|I_h u\|_{L^\infty(\Om)} \\
	\leq &\ c h^{-d/2} \|R_h u - u\|_{L^2(\Om)} + c h^{-d/2} \|u - I_h u\|_{L^2(\Om)} + \|I_h u\|_{L^\infty(\Om)} \\
	\leq &\ c h^{2-d/2} \|u\|_{H^2(\Om)} + c h^{2-d/2} \|u\|_{H^2(\Om)} + \|u\|_{L^\infty(\Om)} .
	\end{aligned}
	\end{equation}
	Note that, there holds $\|\ga(R_h u)\|_{L^\infty(\Ga)} \leq \|R_h u\|_{L^\infty(\Om)}$. \bbk We recall that by assumption $u$ has a finite $L^\infty$-norm, cf.~\eqref{eq:sufficient regularity}. \ebk 
	
	Using the $L^\infty$ norm bound on the Ritz map \eqref{eq:Ritz map Linfty bound} and the local Lipschitz continuity of $W'$, via \eqref{eq:star and H_h^-1 norm definition}, we obtain that the defect in $w$ is bounded by
	\begin{align*}
	\|d_h^w\|_{\ast,h} \leq c h^2 (\|u\|_{L^\infty(\Om)} + \|u\|_{H^2(\Om)} + \|\ga u\|_{H^2(\Ga)} + \|w\|_{H^2(\Om)} + \|\ga w\|_{H^2(\Ga)} ) 
	\end{align*}
	
	The time derivative of the defect $d_h^w$ is bounded by the same techniques. We take the time derivative of the equation \eqref{eq:defect w - pre estimates}, use that $\vphi_h \in V_h$ is time independent and that $\d / \d t \, (R_h u ) = R_h \dot u$, the $L^\infty$ bounds on the Ritz map of $\dot u$ (obtained analogously as for $u$), and then use the same estimates as above to obtain
	\begin{align*}
	\|\dot d_h^w\|_{\ast,h} \leq c h^2 (\|\dot u\|_{L^\infty(\Om)} + \|\dot u\|_{H^2(\Om)} + \|\ga \dot u\|_{H^2(\Ga)} + \|\dot w\|_{H^2(\Om)} + \|\ga \dot w\|_{H^2(\Ga)} ) .
	\end{align*}
\end{proof}

\section{Proof of Theorem~\ref{theorem:semidiscrete convergence}}
\label{section:proof of main theorem}

Combining the results of the two previous section we now prove the semi-discrete convergence theorem.

\begin{proof}[Proof of Theorem~\ref{theorem:semidiscrete convergence}]
	The proof combines the results of the previous two sections on stability and consistency. 
	
	The error between the lifted numerical solutions $u_h^\ell(\cdot,t)$ and $w_h^\ell(\cdot,t)$ and the exact solutions $u(\cdot,t)$ and $w(\cdot,t)$ decomposes to (omitting the time $t$):
	\begin{align*}
	u_h^\ell - u = &\ (u_h - \widetilde{R}_h u)^\ell + (R_h u - u) = (e_h^u)^\ell + (R_h u - u) , \\
	w_h^\ell - w = &\ (w_h - \widetilde{R}_h w)^\ell + (R_h w - w) = (e_h^w)^\ell + (R_h w - w) ,
	\end{align*}
	where the second equalities follow upon recalling the definitions of the errors from Section~\ref{section:error equations}. A similar decomposition holds for the time derivatives as well.
	
	The first terms are estimated by combining the stability, Proposition~\ref{proposition:stability}, with the bounds for the defects, Proposition~\ref{proposition:defect bounds}, and the fact that  $e_h^u(0)=0$ and $e_h^w(0)=0$, which altogether immediately implies the assumed bounds in \eqref{eq:defect bounds - assumed}. Altogether, using a norm equivalence \eqref{eq:norm equivalence}, we obtain, for $0 \leq t \leq T$,
	\begin{equation*}
	\begin{aligned}
	&\ \bigg( |(e_h^u)^\ell(\cdot,t)|^2 + |(e_h^w)^\ell(\cdot,t)|^2 + \int_0^t |(\dot e_h^u)^\ell(\cdot,s) |^2 \d s \bigg)^{1/2}  \\
	\leq &\ \bigg( \|(e_h^u)^\ell(\cdot,t)\|^2 + \|(e_h^w)^\ell(\cdot,t)\|^2 + \int_0^t \|(\dot e_h^u)^\ell(\cdot,s) \|^2 \d s \bigg)^{1/2} \leq C h^2,
	\end{aligned}
	\end{equation*}
	where the first inequality holds by the natural estimate $|\cdot| \leq \|\cdot\|$.
	
	The second terms are estimated directly by the Ritz map error estimates of Lemma~\ref{lemma:Ritz error est}. By translating back from the abstract functional analytic setting, in the stronger $V$ (i.e.~$H^1$) norm and in the weaker $H$ (i.e.~$L^2$) norm, we respectively obtain
	\begin{align*}
	\|R_h u - u\| = &\ \|R_h u - u\|_{H^1(\Om)} + \|\ga (R_h u - u)\|_{H^1(\Ga)} \leq c h , \\
	|R_h u - u| = &\ \|R_h u - u\|_{L^2(\Om)} + \|\ga (R_h u - u)\|_{L^2(\Ga)} \leq c h^2 . \\
	\end{align*}
	For $w$ we have the same bounds.

	Combining these estimates yields the stated convergence result.
\end{proof}

\section{Linearly implicit backward difference time discretisation}
\label{section:BDF}

The semi-discrete problem \eqref{eq:semi-discrete problem} is first rewritten in the matrix--vector form, with $\bfu(t)$ and $\bfw(t)$ collecting the nodal values of the finite element functions $u_h(\cdot,t)$ and $w_h(\cdot,t)$ in $V_h$, respectively,
\begin{subequations}
	\label{eq:CH matrix vector form}
	\begin{align}
	\bfM \dot\bfu(t) - \bfA \bfw(t) = &\ 0 , \\
	\bfM \bfw(t) + \bfA \bfu(t) = &\ \bfW'(\bfu(t)) .
	\end{align}
\end{subequations}
Here $\bfA$ and $\bfM$ denote the stiffness and mass matrix given through \eqref{eq:discrete bilinear forms}, while the vector $\bfW'(\bfu(t))$ is given via the right-hand side of \eqref{eq:semi-discrete problem - b}.

As a time discretisation of the system \eqref{eq:CH matrix vector form}, we consider the linearly implicit $q$-step \emph{backward differentiation formulae} (BDF). 
For a step size $\tau>0$, and with $t_n = n \tau \leq T$, we determine the approximations to the variables $\bfu^n$ to $\bfu(t_n)$ and $\bfw^n$ to $\bfw(t_n)$ by the fully discrete system of \emph{linear} equations, for $n \geq q$,
\begin{subequations}
	\label{eq:CH BDF}
	\begin{align}
	\label{eq:CH BDF - 1}
	\bfM \dot\bfu^n - \bfA \bfw^n = &\ 0 , \\
	\label{eq:CH BDF - 2}
	\bfM \bfw^n + \bfA \bfu^n = &\ \bfW'(\widetilde{\bfu}^n) ,
	\end{align}
\end{subequations}
where the discretised time derivative is determined by
\begin{equation}
\label{eq:backward differences def}
\dot \bfu^n = \frac{1}{\tau} \sum_{j=0}^q \delta_j \bfu^{n-j} , \qquad n \geq q ,
\end{equation}
while the non-linear term uses an extrapolated value, and is given by:
\begin{equation*}
\bfW'(\widetilde{\bfu}^n) := \bfW' \bigg(\sum_{j=0}^{s-1} \gamma_j \,\bfu^{n - 1 -j} \bigg) , \qquad n \geq s .
\end{equation*}
The starting values $\bfu^i$ ($i=0,\dotsc,q-1$) are assumed to be given. They can be precomputed using either a lower order method with smaller step sizes, or an implicit Runge--Kutta method. The initial values $\bfw^i$ ($i=0,\dotsc,q-1$) are computed from the already obtained $\bfu^i$.

The method is determined by its coefficients, given by $\delta(\zeta)=\sum_{j=0}^q \delta_j \zeta^j=\sum_{\ell=1}^q \frac{1}{\ell}(1-\zeta)^\ell$ and $\gamma(\zeta) = \sum_{j=0}^{q-1} \gamma_j \zeta^j = (1 - (1-\zeta)^q)/\zeta$. 
The classical BDF method is known to be zero-stable for $q\leq6$ and to have order $q$; see \cite[Chapter~V]{HairerWannerII}.
This order is retained by the linearly implicit variant using the above coefficients $\gamma_j$; 
cf.~\cite{AkrivisLubich_quasilinBDF,AkrivisLiLubich_quasilinBDF}.

The anti-symmetric structure of the system \eqref{eq:semi-discrete problem}, is preserved by the above time discretisation, and can be observed in \eqref{eq:CH BDF}. Using the $G$-stability theory of \cite{Dahlquist} and the multiplier technique of \cite{NevanlinnaOdeh}, the energy estimates used in the proof of Proposition~\ref{proposition:stability} can be transferred to linearly implicit BDF full discretisations (up to order 5). 

Therefore, we strongly expect that  Proposition~\ref{proposition:stability} translates to the fully discrete case, and so does the convergence result Theorem~\ref{theorem:semidiscrete convergence}, with classical convergence order in time. 
The successful application of these techniques to the analogous (linearly implicit) BDF discretisation applied to evolving surface PDEs, e.g.~\cite{LubichMansourVenkataraman_bdsurf,ALE2,KovacsPower_quasilinear} showing optimal-order error bounds for various problems on evolving surfaces, strengthens this statement. 
Linearly implicit BDF methods were also analysed for various geometric surface flows: for $H^1$-regularised surface flows \cite{soldrivenBDF}, and for mean curvature flow \cite{MCF}, both proving optimal-order error bounds for full discretisations, using the above mentioned techniques and energy estimates testing with the time derivative of the error.

\section{Numerical experiments}
\label{section:numerics}

In this section we present some numerical experiments to illustrate our theoretical results. We consider the Cahn--Hilliard equation with Cahn--Hilliard-type dynamic boundary conditions \eqref{eq:CH system} in a disk $\Om$ with its boundary $\Ga$. We present two numerical experiments: a convergence test, and a numerical test reporting on the evolution of the $u$ component when the equation is started from random initial data. In both cases for the spatial discretization we use linear bulk--surface finite elements, while for time discretization we use the 2- or 3-step BDF method, cf.~\eqref{eq:CH BDF}. 

\subsection{Convergence test}

For the convergence experiment, additional inhomogeneities are added to each equation in \eqref{eq:CH system}, such that the exact solution is known to be $u(x,t) = w(x,t) = e^{-t} x_1 x_2$. The experiment is performed with the double-well potential both in the bulk and on the surface, i.e.~with the non-linearities $W_\Om(u) = W_\Ga(u) = \tfrac{1}{4}(u^2-1)^2$. \bbk Since, the exact solution is known, we used the interpolation of the exact initial data. \ebk 

The domain $\Om$ is the unit disk, and the final time is $T=1$. For this experiment we used a sequence of time step sizes $\tau = (0.05$, $0.025$, $0.0125$, $0.005$, $0.0025$, $0.00125)$ (having an approximate ratio of $2$), and a sequence of initial meshes with degrees of freedom  $2^k\cdot10 $ for $k=1,\ldots, 8$.

\begin{figure}[htbp]
	\includegraphics[width=\textwidth,trim={0 0 0 0},clip]{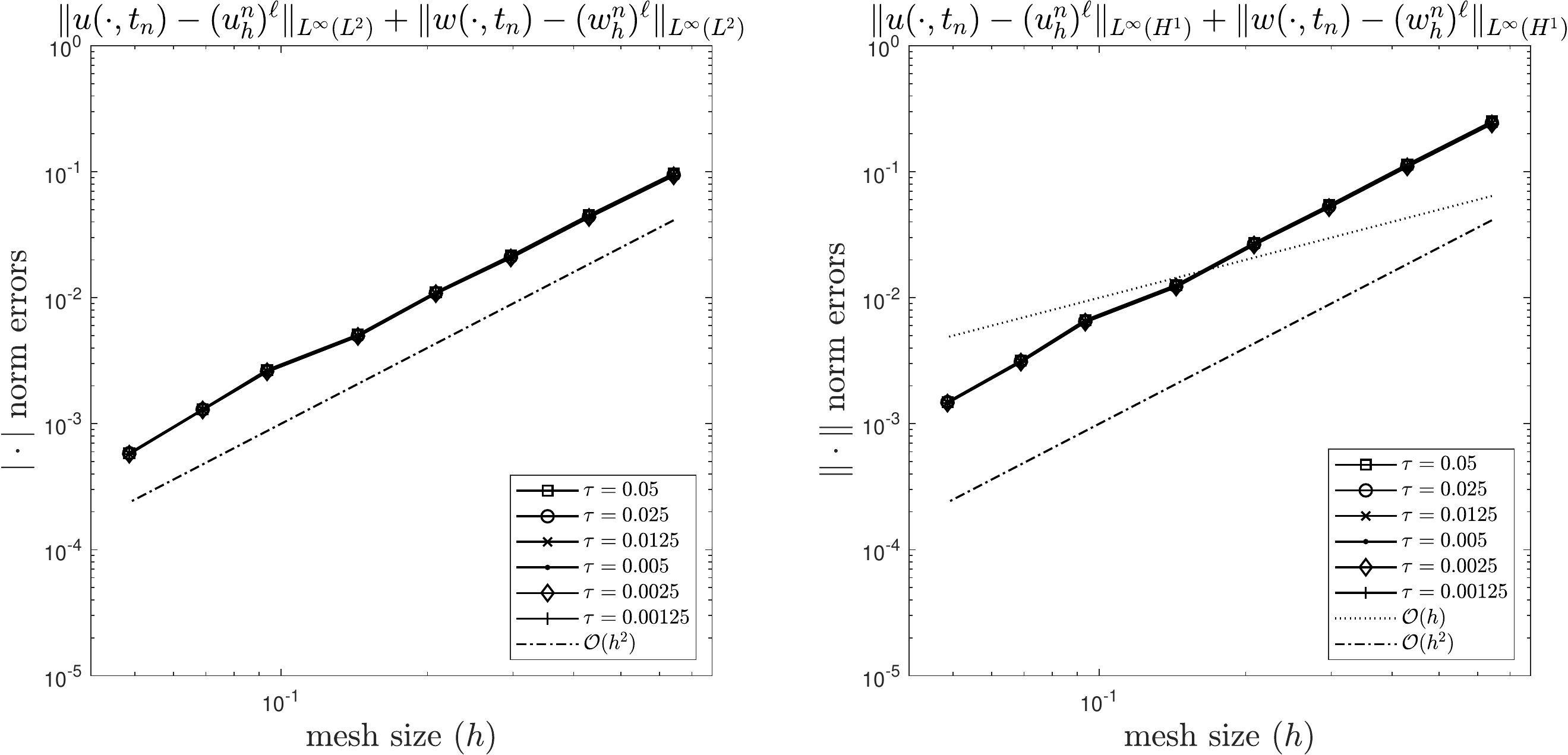}
	\caption{Spatial convergence plots for the linear bulk--surface FEM / BDF3 approximation to the non-linear Cahn--Hilliard equation with Cahn--Hilliard-type dynamic boundary conditions with free energy potentials $W_\Om(u) = W_\Ga(u) = \tfrac{1}{4}(u^2-1)^2$.}
	\label{figure:convplot_CH_nonlinear}
\end{figure}

In Figure~\ref{figure:convplot_CH_nonlinear} we report on the bulk--surface \bbk $L^\infty([0,T],L^2)$ \ebk norm errors (left) and \bbk $L^\infty([0,T],H^1)$ \ebk norm errors (right) between the (linear bulk--surface FEM / BDF3) numerical approximation and the exact solution for both variables, i.e.~both the bulk and the surface error for both variables $u$ and $w$. 
The logarithmic plots show the errors against the mesh width $h$, the lines marked with different symbols correspond to different time step sizes.

In Figure~\ref{figure:convplot_CH_nonlinear} we can observe the spatial discretisation error dominates,  and matches the order of convergence of our theoretical results (note the reference lines). Note that, however, in the $H^1$ norm we observe a better convergence rate than the linear convergence order proven in Theorem~\ref{theorem:semidiscrete convergence}.


\subsection{The Cahn--Hilliard equation with dynamic boundary conditions and a double-well potential}

An illustration of the phenomena of phase separation described  by the Cahn--Hilliard equation with dynamic Cahn--Hilliard boundary conditions is shown in Figure~\ref{figure:solution}. We consider  a non-linear Cahn--Hilliard equation with potentials $W_\Om(u) = W_\Ga(u) = 10(u^2-1)^2$. We again used the linear finite element method and the linearly implicit BDF method of order $2$. For our plots shown in the figure we generated random initial data $u^0 \in \{ -1,1 \}$,  using the disk with radius 10 as a domain, a mesh with 640 nodes, and a time step size $\tau = 0.00125$. \bbk The other required initial data $\bfu^1$ is computed by a linearly implicit backward Euler step. \ebk 
Each subplot in Figure~\ref{figure:solution} shows the solution $u$ of the problem, i.e.~the phase separation, at the displayed times. The shown colorbar is valid for every subplot.

\begin{figure}[htbp]
	\includegraphics[width=1\textwidth,trim={200 0 500 0},clip]{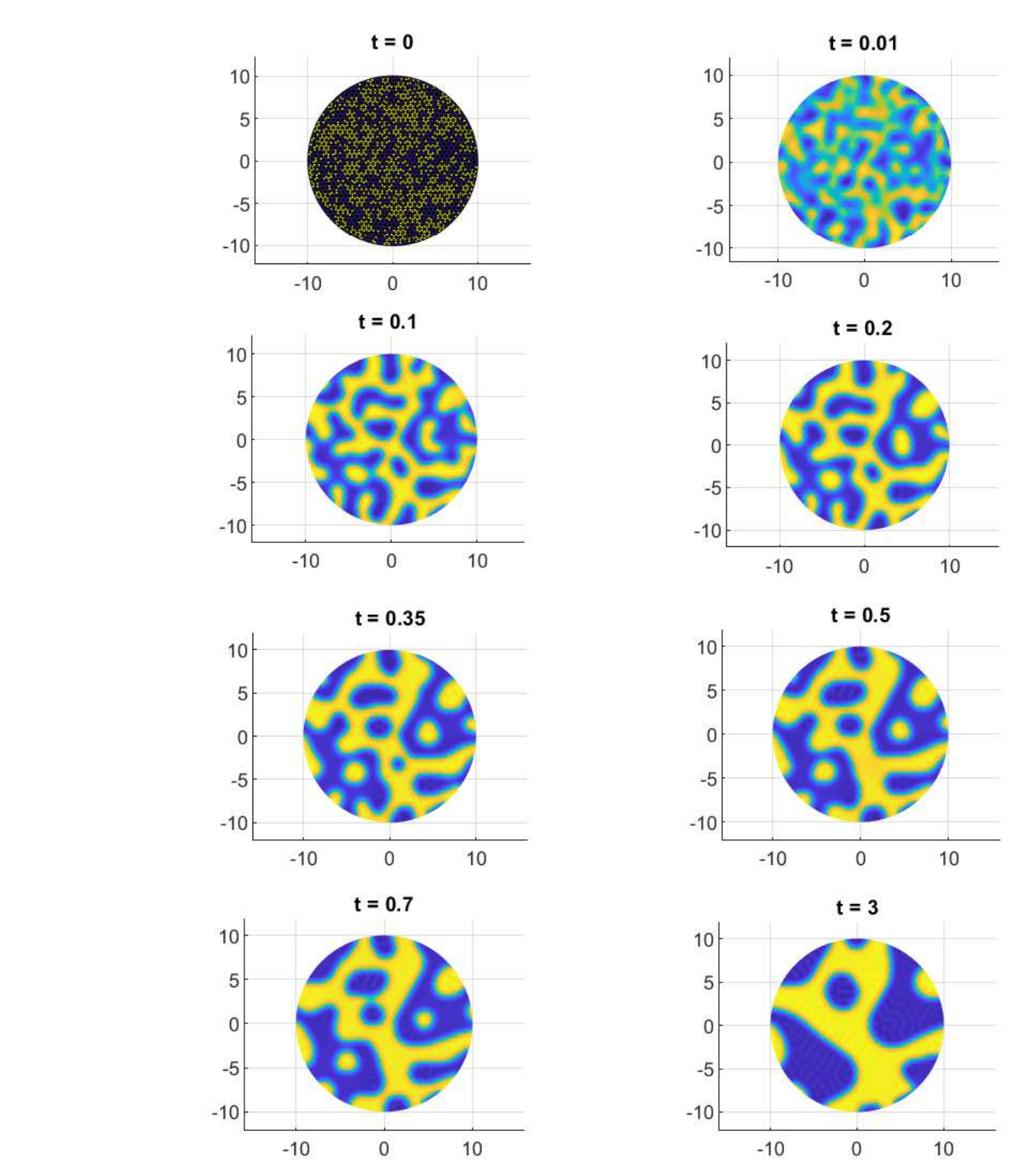}
	\caption{Evolution of the $u_h$ component of the (linear bulk--surface FEM / BDF2) approximation to the non-linear Cahn--Hilliard equation with Cahn--Hilliard-type dynamic boundary conditions, and with free energy potentials $W_\Om(u) = W_\Ga(u) = 10(u^2-1)^2$.}
	\label{figure:solution}
\end{figure}

\section*{Acknowledgements}
We thank Christian Lubich for helpful discussions, in particular on initial values. We also thank Cedric Beschle for the careful reading of the manuscript.

The manuscript was mostly written when both authors have been working at th University of T\"ubingen. We gratefully acknowledge their support.

We thank two Referees whose comments have helped us to improve the presentation of the paper.

The work of Bal\'azs Kov\'acs is supported by Deutsche Forschungsgemeinschaft -- Project-ID 258734477 -- SFB 1173, and by the Heisenberg Programme of the Deutsche Forschungsgemeinschaft (DFG, German Research Foundation) -- Project-ID 446431602.

\bibliographystyle{IMANUM-BIB}
\bibliography{CHdynbc_references}

\end{document}